\documentclass[11 pt, reqno]{amsart}
\usepackage{CJKutf8}
\usepackage{mathrsfs}
\usepackage{graphicx}
\usepackage{tikz}
\usepackage{amsmath}
\usepackage{amsfonts}
\usepackage{amssymb}
\usepackage{hyperref}
\newtheorem{thm}{Theorem}
\newtheorem{lem}[thm]{Lemma}
\newtheorem{cor}[thm]{Corollary}

\newtheorem{pro}[thm]{Proposition}

\theoremstyle{definition}
\newtheorem{definition}{Definition}
\newtheorem{remark}{Remark}

\usepackage{amsthm,amsmath,amssymb,url,cite,color}

\def\blue{\textcolor{blue}}
\def\red{\textcolor{red}}

\numberwithin{equation}{section}

\def\maj{\operatorname{maj}}
\def\exc{\operatorname{exc}}
\def\Exc{\operatorname{EXC}}
\def\fmaj{\operatorname{fmaj}}

\def\fexc{\operatorname{fexc}}
\def\asc{\operatorname{asc}}
\def\des{\operatorname{des}}
\def\Des{\operatorname{DES}}
\def\rmaj{\operatorname{rmaj}}
\def\fmaf{\operatorname{fmaf}}
\def\maf{\operatorname{maf}}
\def\inv{\operatorname{inv}}
\def\lec{\operatorname{lec}}
\def\flec{\operatorname{flec}}
\def\pix{\operatorname{pix}}

\def\wt{\operatorname{wt}}
\def\ps{{\bf ps}}
\def\cs{\operatorname{csum}}
\def\inc{\operatorname{inc}}
\def\Dex{\operatorname{DEX}}
\def\DEX{\operatorname{DEX}}
\def\Rise{\operatorname{RISE}}
\def\Rec{\operatorname{REC}}
\def\Dec{\operatorname{DEC}}
\def\Inc{\operatorname{INC}}
\def\Dex{\operatorname{DEX}}

\def\TB{\operatorname{TB}}

\def\fix{\operatorname{fix}}
\def\col{\operatorname{col}}

\def\length{\operatorname{length}}
\def\Abs{\operatorname{Abs}}
\def\ai{\operatorname{ai}}
\def\rix{\operatorname{rix}}
\def\fdes{\operatorname{fdes}}
\def\st{\operatorname{st}}

\def\Com{\operatorname{Com}}

\def\Par{\operatorname{Par}}

\def\x{{\bf x}}
\def\k{\kappa}
\def\v{{\bf v}}

\def\S{\mathfrak{S}}

\def\o{\mathfrak{o}}
\def\N{\mathbb N}

\def\Q{\mathbb Q}
\def\P{\mathbb P}
\begin{document}

\title[Colored Eulerian quasisymmetric functions]{On some colored Eulerian quasisymmetric \\ functions}

\begin{abstract}
Recently, Hyatt introduced some colored Eulerian quasisymmetric function to study the joint distribution of excedance number and major index on colored permutation groups. 
We show how Hyatt's generating function formula for the fixed point colored Eulerian quasisymmetric functions  can be deduced from the Decrease value theorem of Foata and Han. Using this generating function formula, we prove two symmetric function generalizations of the Chung-Graham-Knuth symmetrical Eulerian identity for  some flag Eulerian quasisymmetric functions, which are specialized to the flag excedance numbers. Combinatorial proofs of those symmetrical identities are also constructed. 

We also study some other properties of the flag Eulerian quasisymmetric functions. In particular, we confirm a recent conjecture of Mongelli [Journal of  Combinatorial Theory, Series A, 120 (2013) 1216--1234] about the unimodality of the generating function of the flag excedances over the type B derangements.
Moreover, colored versions of the hook factorization and  admissible inversions of permutations are found, as well as a new recurrence formula for the $(\maj-\exc,\fexc)$-$q$-Eulerian polynomials. 

We introduce a colored analog of Rawlings major index on colored permutations and obtain an interpretation of the colored Eulerian quasisymmetric functions as sums of some fundamental quasisymmetric functions related with them, by applying Stanley's $P$-partition theory and a decomposition of the Chromatic quasisymmetric functions due to Shareshian and Wachs.
\end{abstract}

\keywords{Statistics; Colored permutations; Colored Eulerian quasisymmetric functions; Symmetrical Eulerian identities; Hook factorization; admissible inversions; Rawlings major index}

\author{Zhicong Lin}
\address[Zhicong Lin]{Universit\'{e} de Lyon; Universit\'{e} Lyon 1; Institut Camille Jordan; UMR 5208 du CNRS; 43, boulevard du 11 novembre 1918, F-69622 Villeurbanne Cedex, France}
\email{lin@math.univ-lyon1.fr}


\maketitle

\tableofcontents

\section{Introduction}

%
A permutation of $[n]:=\{1,2,\ldots,n\}$ is a bijection $\pi : [n]\rightarrow[n]$. Let $\S_n$ denote the set of permutations of $[n]$. For
each $\pi\in\S_n$, a value $i$, $1\leq i\leq n-1$, is an \emph{excedance} (resp.~\emph{descent}) of $\pi$ if $\pi(i)>i$ (resp.~$\pi(i)>\pi(i+1)$). Denote by $\exc(\pi)$ and $\des(\pi)$ the number of excedances and descents of $\pi$, respectively. The classical {\em Eulerian number}, which we will denote by $A_{n,k}$, counts the number of  permutations in $\S_n$ with $k$ excedances (or $k$ descents). The Eulerian numbers arise in a variety of contexts in mathematics and have many other remarkable properties; see~\cite{fo} for a informative historical introduction.  
 
There are not so many combinatorial identities for Eulerian numbers comparing with other sequences such as binomial coefficients or Stirling numbers. Nevertheless, Chung, Graham and Knuth~\cite{cgk} proved the following  symmetrical identity:
\begin{equation}\label{cgk:symiden}
\sum_{k\geq 1}{a+b\choose k}A_{k,a-1}=\sum_{k\geq 1}{a+b\choose k}A_{k,b-1}
\end{equation}
for $a,b\geq1$.  Recall that  the {\em major index}, $\maj(\pi)$, of a permutation $\pi\in\S_n$ is the sum of all the descents of $\pi$, i.e.,
$\maj(\pi):=\sum_{\pi_i>\pi_{i+1}}i$. Define the {\em$q$-Eulerian numbers} $A_{n,k}(q)$ by 
$
A_{n,k}(q):=\sum_{\pi} q^{(\maj-\exc)\pi}
$
summed over all permutations $\pi\in\S_n$ with $\exc(\pi)=k$.
 As usual, the {\em$q$-shifted factorial} $(a;q)_n :=\prod_{i=0}^{n-1}(1-aq^i)$ and the {\em $q$-binomial coefficients} ${n\brack k}_q$ are defined by 
$
{n\brack k}_q:=\frac{(q;q)_n }{(q;q)_{n-k}(q;q)_k}.
$
 A $q$-analog of~\eqref{cgk:symiden}  involving both 
 ${n\brack k}_q$ and  $A_{n,k}(q)$ was proved in~\cite{hlz}, by making use of an exponential generating function formula  due to Shareshian and Wachs~\cite{sw}:
 \begin{equation}\label{fixversion}
\sum_{n\geq0}A_n(t,q)\frac{z^n}{(q;q)_n}=\frac{(1-t)e(z;q)}{e(tz;q)-te(z;q)},
\end{equation}
where $A_n(t,q)$ is the {\em $q$-Eulerian polynomial} $\sum_{k=0}^{n-1}A_{n,k}(q)t^k$ and  $e(z;q)$ is the 
{\em $q$-exponential function} $\sum_{n\geq 0}\frac{z^n}{(q;q)_n}.$
 Shareshian and Wachs obtained~\eqref{fixversion} by introducing certain quasisymmetric functions (turn out to be symmetric functions), called {\em Eulerian quasisymmetric functions}, such that applying the stable principal specialization yields the $q$-Eulerian numbers.
%

Let $l$ be a fixed positive integer throughout this paper. Now consider the wreath product $C_l\wr\S_n$ of the cyclic group $C_l$ of order $l$ by the symmetric group $\S_n$ of order $n$.  The group $C_l\wr\S_n$ is also known as the {\em colored permutation group} and reduces to the permutation group $\S_n$ when $l=1$. It is worth to note that, Foata and Han~\cite{fh2} studied various statistics on words and obtain a factorial generating function formula implies~\eqref{fixversion} for the quadruple distribution, involving the number of fixed points, the excedance number, the descent number and the major index, on permutations and further generalized to colored permutations~\cite{fh3}. Recently, in order to generalize~\eqref{fixversion} to colored permutation groups, Hyatt~\cite{hy} introduced some {\em colored Eulerian quasisymmetric functions} (actually symmetric functions), which are generalizations of the Eulerian quasisymmetric functions. The starting point for the present paper is the attempt to obtain a symmetric function generalization of~\eqref{cgk:symiden} for colored permutation groups.

The most refined version of colored Eulerian quasisymmetric functions are {\em cv-cycle type colored Eulerian quasisymmetric functions} $Q_{\check{\lambda},k}$, where $\check{\lambda}$ is a particular cv-cycle type. They are defined by first associating a fundamental quasisymmetric function with each colored permutation and then summing these fundamental quasisymmetric functions over colored permutations with cv-cycle type $\check{\lambda}$ and $k$ excedances. The precise definition of $Q_{\check{\lambda},k}$ is given in Section~\ref{st:color}.  It was announced in~\cite{hy} that $Q_{\check{\lambda},k}$ is in fact a symmetric function. This follows from the colored ornament interpretation of $Q_{\check{\lambda},k}$ and the plethysm inversion formula. But more importantly, we will give a combinatorial proof of this fact which is needed in the bijective proof of Theorem~\ref{sym:iden1} below.

Another interesting Eulerian quasisymmetric function is the {\em fixed point colored Eulerian quasisymmetric function} $Q_{n,k,\Vec{\alpha},\Vec{\beta}}$, for $\Vec{\alpha}\in\N^l$ and $\Vec{\beta}\in\N^{l-1}$, which can be defined as certain sum of $Q_{\check{\lambda},k}$. The main result in~\cite{hy} is a generating function formula (see Theorem~\ref{hyatt}) for  $Q_{n,k,\Vec{\alpha},\Vec{\beta}}$, which when applying the stable 
 principal specialization would yield a generalization of~\eqref{fixversion} for the joint distribution of excedance number and major index on colored permutations. 
 This generating function formula was obtained through three main steps. 
Firstly, a colored analog of the Gessel-Reutenauer bijection~\cite{gr} is used to give the colored ornaments characterization of $Q_{\check{\lambda},k}$; secondly, the Lyndon decomposition is used to give the colored banners characterization of $Q_{\check{\lambda},k}$; finally, 
the generating function formula is derived by establishing a recurrence formula using the interpretation of $Q_{\check{\lambda},k}$ as colored banners. The recurrence formula in step 3 is obtained through a complicated generalization of a bijection of Shareshian-Wachs~\cite{sw}, so it would be reasonable to expect a simpler approach. We will show how this generating function formula (actually the step 3)  can be deduced directly from the Decrease value theorem on words due to Foata and Han~\cite{fh4}. 
 
 We modify the fixed point Eulerian quasisymmetric functions to some $Q_{n,k,j}$ that we call {\em flag Eulerian quasisymmetric functions}, which are also generalizations of  Shareshian and Wachs' Eulerian quasisymmetric functions and would specialize to  the {\em flag excedance numbers} studied in~\cite{bg,fh3}. The generating function formula for $Q_{n,k,j}$ follows easily from the generating function formula of $Q_{n,k,\Vec{\alpha},\Vec{\beta}}$, and is used to prove the following two symmetric function generalizations of~\eqref{cgk:symiden} involving both the {\em complete homogeneous symmetric functions} $h_n$ and the flag Eulerian quasisymmetric functions $Q_{n,k,j}$.

\begin{thm} \label{sym:iden1}
For $a,b\geq1$ and $j\geq0$ such that $a+b+1=l(n-j)$, 
\begin{equation*}
\sum_{i\geq0}h_{i}Q_{n-i,a,j}=\sum_{i\geq0}h_iQ_{n-i,b,j}.
\end{equation*}
 \end{thm}
 
 \begin{thm}\label{sym:iden2}
 Let $Q_{n,k}=\sum_{j}Q_{n,k,j}$. 
  For $a,b\geq1$ such that $a+b=ln$,
\begin{equation*}
\sum_{i=0}^{n-1}h_{i}Q_{n-i,a-1}=\sum_{i=0}^{n-1}h_iQ_{n-i,b-1}.
\end{equation*}
 \end{thm}
 
We will construct bijective proofs of those two generalized symmetrical identities, one of which leads to a new interesting approach to the step~3 of~\cite[Theorem~1.2]{sw}. Define the {\em fixed point colored $q$-Eulerian numbers} by
\begin{equation}\label{colored:eulnum}
A_{n,k,j}^{(l)}(q):=\sum_{\pi} q^{(\maj-\exc)\pi}
\end{equation}
summed over all colored permutations $\pi\in C_l\wr\S_n$ with $k$ flag excedances and $j$ fixed points.
Applying the stable principle specialization to the two identities in Theorem~\ref{sym:iden1} and~\ref{sym:iden2} then yields two symmetrical identities for $A_{n,k,j}^{(l)}(q)$,  which are colored analog of two $q$-Eulerian symmetrical identities appeared  in~\cite{cg,hlz}. A new recurrence formula for the colored $q$-Eulerian numbers $A_{n,k,j}^{(l)}(q)$ is also proved.  Note that Steingr\'imsson~\cite{ste} has already generalized various joint pairs of statistics on permutations to colored permutations. More recently, Faliharimalala and Zeng~\cite{fz} introduced a Mahonian statistic $\fmaf$ on colored permutations and extend the triple statistic $(\fix,\exc,\maf)$, a triple that is equidistributed with $(\fix,\exc,\maj)$ studied in~\cite{fo5}, to the colored permutations.
In the same vein, we find  generalizations of {\em Gessel's hook factorizations}~\cite{ge} and the {\em admissible inversion} statistic introduced by Linusson, Shareshian and Wachs~\cite{lsw}, which enable us to obtain two new interpretations for $A_{n,k,j}^{(l)}(q)$.
 
 Let $Q_{n,k,\Vec{\beta}}$ be the  colored Eulerian quasisymmetric function (does not take the fixed points into account) defined by 
 $
 Q_{n,k,\Vec{\beta}}:=\sum_{\Vec{\alpha}}Q_{n,k,\Vec{\alpha},\Vec{\beta}}
 $.
 We obtain a new interpretation of $Q_{n,k,\Vec{\beta}}$ as sums of some fundamental quasisymmetric functions related with an analogue of Rawlings major index~\cite{ra} on colored permutations. This is established by applying the $P$-partition theory and a decomposition of the Chromatic quasisymmetric functions due to Shareshian and Wachs~\cite{sw3}. A consequence of this new interpretation is another interpretation for the {\em colored $q$-Eulerian numbers} $A_{n,k}^{(l)}(q)$ defined as
 \begin{equation}\label{coloredn:eulnum}
 A_{n,k}^{(l)}(q):=\sum_{j\geq0}A_{n,k,j}^{(l)}(q).
 \end{equation}
 

This paper is organized as follows. In section~\ref{CEQF}, we recall some statistics on colored permutations and the definition of the cv-cycle type colored Eulerian quasisymmetric function $Q_{\check{\lambda},j}$. We prove that $Q_{\check{\lambda},j}$ is a symmetric function using the interpretation of colored ornaments and state Hyatt's formula for the generating function of the fixed point colored Eulerian quasisymmetric functions. In section~\ref{dec-val-thm}, we show how to deduce Hyatt's generating function formula from the decrease value theorem. In section~\ref{flag:euler}, we introduce the flag Eulerian quasisymmetric function and prove Theorem~\ref{sym:iden1} and~\ref{sym:iden2}, both analytically and combinatorially. Some other properties of the flag Eulerian quasisymmetric functions and the fixed point colored $q$-Eulerian numbers $A_{n,k,j}^{(l)}(q)$ are also proved. 
In section~\ref{color:rawling}, we introduce a colored analog of Rawlings major index  and obtain a new interpretation for $Q_{n,k,\Vec{\beta}}$ and therefore an another interpretation for $A_{n,k}^{(l)}(q)$. 

\medskip 

\noindent {\bf Notations on quasisymmetric functions.} We collect here the definitions and some facts about Gessel's quasisymmetric functions that will be  used in the rest of this paper; a good reference is~\cite[Chapter~7]{st2}.
Given a subset $S$ of $[n-1]$, define the \emph{fundamental quasisymmetric function} $F_{n,S}$ by 
$$
F_{n,S}=F_{n,S}(\x):=\sum_{i_1\geq\cdots\geq i_n\geq1\atop j\in S\Rightarrow i_j>i_{j+1}}x_{i_1}\cdots x_{i_n}.
$$
If $S=\emptyset$ then $F_{n,S}$ is the complete homogeneous symmetric function $h_n$ and if $S=[n-1]$ then $F_{n,S}$ is the {\em elementary symmetric function} $e_n$.
Define $\omega$ to be  the involution on the ring of quasisymmetric functions that maps $F_{n,S}$ to $F_{n,[n-1]\setminus S}$, which extends the involution on the ring of symmetric functions that takes $h_n$ to $e_n$. 

The {\em stable principal specialization} $\ps$ is the ring homomorphism  from the ring of symmetric functions to the ring of formal power series in the variable $q$, defined by 
$$
\ps(x_i)=q^{i-1}.
$$
The following property of $\ps$ is known (see~\cite[Lemma~5.2]{gr})
\begin{equation}\label{quasi-ps}
\ps(F_{n,S})=\frac{q^{\sum_{i\in S}i}}{(q;q)_n}.
\end{equation}
In particular, $\ps(h_n)=1/(q;q)_n$.

\section{Colored Eulerian quasisymmetric functions}
\label{CEQF}
\subsection{Statistics on colored permutation groups}\label{st:color}
We shall recall the definition of the colored Eulerian quasisymmetric functions introduced in~\cite{hy}.
Consider the following  set of {\em$l$-colored integers} from $1$ to $n$
$$
[n]^l:=\left\{1^{0}, 1^{1}, \ldots, 1^{l-1},
2^{0}, 2^{1}, \ldots, 2^{l-1},\ldots,
n^0, n^1, \ldots, n^{l-1}\right\}. 
$$
If $\pi$ is a word over $[n]^l$, we use $\pi_i$ to denote the $i$th letter of $\pi$. We let $|\pi_i|$ denote the positive integer obtained by removing the superscript, and let $\epsilon_i\in\{0,1,\ldots,l-1\}$ denote the superscript, or color, of the $i$th letter of the word. If $\pi$ is a word of length $m$ over $[n]^l$, we 
denote by $|\pi|$ the word
$$|\pi|:=|\pi_1||\pi_2|\cdots|\pi_m|.$$
In one-line notation, the {\em colored permutation group} $C_l\wr\S_n$ can be viewed as the set of words over $[n]^l$ defined by 
$$
\pi\in C_l\wr\S_n\Leftrightarrow|\pi|\in\S_n.
$$

Now, the {\em descent number}, $\des(\pi)$, the {\em excedance number}, $\exc(\pi)$, and the {\em major index}, $\maj(\pi)$, of a colored permutation $\pi\in C_l\wr\S_n$ are defined as  follows:
\begin{align*}
&\Des(\pi):=\{j\in[n-1] : \pi_j>\pi_{j+1}\},\\
&\des(\pi):=|\Des(\pi)|,\quad \maj(\pi):=\sum_{j\in\Des(\pi)} j\\
&\Exc(\pi):=\{j\in[n] : \pi_j>j^0\},\quad \exc(\pi):=|\Exc(\pi)|,
\end{align*}
where we use the following {\em color order}
$$
\mathcal{E}:=\left\{1^{l-1}<2^{l-1}<\cdots<n^{l-1}
<1^{l-2}<2^{l-2}<\cdots<n^{l-2}<\cdots
<1^0<2^0<\cdots<n^0\right\}.
$$
Also, for $0\leq k\leq l-1$, the $k$-th color fixed point number $\fix_k(\pi)$ and the $k$-th color number $\col_k(\pi)$ are defined by
$$
\fix_k(\pi):=|\{j\in[n] : \pi_j=j^k\}|\quad\text{and}\quad\col_k(\pi):=|\{j\in[n] : \epsilon_j=k\}|.
$$
The {\em fixed point vector} $\Vec{\fix}(\pi)\in\N^{l}$ and the {\em color vector} $\Vec{\col}(\pi)\in\N^{l-1}$ are defined by
$$
\Vec{\fix}(\pi):=(\fix_0(\pi),\fix_1(\pi),\ldots,\fix_{l-1}(\pi)), \quad \Vec{\col}(\pi):=(\col_1(\pi),\ldots,\col_{l-1}(\pi))
$$
respectively.
For example, if $\pi=5^2\,2^1\,4^0\,3^2\,1^2\, 6^0\in C_3\wr\S_6$, then $\Des(\pi)=\{3,4\}$, $\des(\pi)=2$, $\exc(\pi)=1$, $\maj(\pi)=7$, $\Vec{\fix}(\pi)=(1,1,0)$ and $\Vec{\col}(\pi)=(1,3)$.

The colored permutations can also be written in cycle form such that $j^{\epsilon_j}$ follows $i^{\epsilon_i}$ means that $\pi_i=j^{\epsilon_j}$.  Continuing with the previous example, we can write it in cycle form as
\begin{equation}\label{cycleform}
\pi=(1^2,5^2)(2^1)(3^2,4^0)(6^0).
\end{equation}

Next we recall the cv-cycle type of a colored permutation $\pi\in C_l\wr\S_n$. Let $\lambda=(\lambda_1\geq\cdots\geq\lambda_i)$ be a partition of $n$. Let $\vec{\beta^1},\ldots,\vec{\beta^i}$ be a sequence of vectors in $\N^{l-1}$ with $|\vec{\beta^j}|\leq\lambda_j$ for $1\leq j\leq i$, where  $|\vec{\beta}|:=\beta_1+\cdots\beta_{l-1}$ for each $\vec{\beta}\in\N^{l-1}$. Consider the multiset of pairs 
\begin{equation}\label{cv-type}
\check{\lambda}=\{(\lambda_1,\vec{\beta^1}),\ldots,(\lambda_i,\vec{\beta^i})\}.
\end{equation}
A permutation $\pi$ is said to have {\em cv-cycle type} $\check{\lambda}(\pi)=\check{\lambda}$ if   each pair $(\lambda_j,\vec{\beta^j})$ corresponds to exactly one cycle of length $\lambda_j$ with color vector 
$\vec{\beta^j}$ in the cycle decomposition of $\pi$. Note that $\Vec{\col}(\pi)=\vec{\beta^1}+\vec{\beta^2}+\cdots+\vec{\beta^i}$ using component wise addition. For example, the permutation 
in~\eqref{cycleform} has $\check{\lambda}(\pi)=\{(2,(0,2)),(2,(0,1)),(1,(1,0)),(1,(0,0))\}$.


We are now ready to give the definition of the main object of this paper.

\begin{definition}[Definition~2.1 of~\cite{hy}]
For any particular cv-cycle type $\check{\lambda}=\{(\lambda_1,\vec{\beta^1}),\ldots,(\lambda_i,\vec{\beta^i})\}$, define the {\em cv-cycle type colored Eulerian quasisymmetric functions} $Q_{\check{\lambda},k}$ by
$$
Q_{\check{\lambda},k}:=\sum_{\pi}F_{n,\DEX(\pi)}
$$
summed over $\pi\in C_l\wr\S_n$ with $\check{\lambda}(\pi)=\check{\lambda}$ and $\exc(\pi)=k$, where $\Dex(\pi)$ is some set value statistic related with $\Des$. We don't need the detailed definition of $\Dex$ in this paper.
Given $\Vec{\alpha}\in\N^l,\Vec{\beta}\in\N^{l-1}$, the {\em fixed point colored Eulerian quasisymmetric functions}  are then defined as
\begin{equation}\label{DEX:int}
Q_{n,k,\Vec{\alpha},\Vec{\beta}}=\sum_{\pi}F_{n,\DEX(\pi)}
\end{equation}
summed over all  $\pi\in C_l\wr\S_n$ such that $\exc(\pi)=k,\Vec{\fix}(\pi)=\Vec{\alpha}$ and $\Vec{\col}(\pi)=\Vec{\beta}$.
\end{definition}

The following specialization of the fixed point colored Eulerian quasisymmetric functions follows from~\cite[Lemma~2.2]{hy} and Eq.~\eqref{quasi-ps}.
\begin{lem}\label{DEX:lem}
For all $n,k,\Vec{\alpha}$ and $\Vec{\beta}$,
\begin{equation}\label{ps:dex}
\ps(Q_{n,k,\Vec{\alpha},\Vec{\beta}})=(q;q)_n^{-1}\sum_{\pi}q^{(\maj-\exc)\pi}
\end{equation}
summed over all  $\pi\in C_l\wr\S_n$ such that $\exc(\pi)=k,\Vec{\fix}(\pi)=\Vec{\alpha}$ and $\Vec{\col}(\pi)=\Vec{\beta}$.
\end{lem}

\subsection{Colored ornaments} We will use the colored ornament interpretation in~\cite{hy}  to prove combinatorially that $Q_{\check{\lambda},k}$ is a symmetric function. 

Let $\mathcal{B}$ be the infinite ordered alphabet given by
\begin{equation}\label{original:order}
\mathcal{B}:=\{1^0<1^1<\cdots<1^{l-1}<\overline{1^0}<2^0<2^1<\cdots<2^{l-1}<\overline{2^0}<3^0<3^1<\cdots\}.
\end{equation}
A letter of the form $u^m$ is said to be $m$-colored and the letter $\overline{u^0}$ is called $0$-colored. 
If $w$ is a word over $\mathcal{B}$, we define the {\em color vector} $\Vec{\col}(w)\in\N^{l-1}$ of $w$ to be 
$$
\Vec{\col}(w):=(\col_1(w),\col_2(w),\ldots,\col_{l-1}(w)),
$$
where $\col_m(w)$ is the number of $m$-colored letters in $w$ for $m=1,\ldots,l-1$. The {\em  absolute value of a letter} is the positive integer obtained by removing any colors or bars, so $|u^m|=|\overline{u^0}|=u$. The {\em weight of a letter} $u^m$ or $\overline{u^0}$ is $x_u$.

We consider the circular word over $\mathcal{B}$. If $w$ is a word on $\mathcal{B}$, we denote $(w)$ the {\em circular word} obtained by placing the letters of $w$ around a circle in a clockwise direction. A circular word $(w)$ is said to be {\em primitive} if the word $w$ can not be written as $w=w'w'\cdots w'$ where $w'$ is some proper subword of $w$. For example, $(\overline{1^0},2^1,1^0,2^1)$ is primitive but $(1^0,2^1,1^0,2^1)$ is not because $1^02^11^02^1=w'w'$ with $w'=1^02^1$.

\begin{definition}[Definition~3.1 of~\cite{hy}]
A {\em colored necklace} is a circular primitive word  $(w)$ over the alphabet $\mathcal{B}$ such that 
\begin{itemize}
\item[(1)]  Every barred letter is followed by a letter of lesser or equal absolute value.
\item[(2)] Every $0$-colored unbarred letter is followed by a letter of greater or equal absolute value.
\item[(3)] Words of length one may not consist of a single barred letter.
\end{itemize}
A {\em colored ornament} is a multiset of colored necklaces.
\end{definition}

The {\em weight} $\wt(R)$ of a ornament $R$ is the product of the weights of the letters of $R$.
Similar to the {\em cv-cycle type}  of a colored permutation, the cv-cycle type $\check{\lambda}(R)$ of a colored ornament $R$ is the multiset 
$$
\check{\lambda}(R)=\{(\lambda_1,\vec{\beta^1}),\ldots,(\lambda_i,\vec{\beta^i})\},
$$
where each pair $(\lambda_j,\vec{\beta^j})$ corresponds to precisely one colored necklace in the ornament $R$ with length $\lambda_j$ and color vector $\vec{\beta^j}$ .

The following colored ornament interpretation of $Q_{\check{\lambda},k}$ was proved by Hyatt~\cite[Corollary~3.3]{hy} through a colored analog of the {\em Gessel-Reutenauer bijection}~\cite{gr}. 
\begin{thm}[Colored ornament interpretation]\label{thm:ornament} 
For all $\check{\lambda}$ and $k$,
$$
Q_{\check{\lambda},k}=\sum_{R}\wt(R)
$$
summed over all colored ornaments of cv-cycle type $\check{\lambda}$ and exactly $k$ barred letters.
\end{thm}

\begin{thm} \label{symfun:cvcycle}
The cv-cycle type Eulerian quasisymmetric function
$Q_{\check{\lambda},k}$ is a symmetric function.
\end{thm}
\begin{proof} We will generalize  the bijective poof of~\cite[Theorem~5.8]{sw}  involving ornaments to the colored case.
For each $k\in\P$, we will construct a bijection $\psi$ between colored necklaces that exchanges the number of occurrences of the value $k$ and $k+1$ in a colored necklace, but preserves the number of occurrences of all other values, the total number of bars and the color vector. The results will then follow from Theorem~\ref{thm:ornament}.

\vskip 0.1in
{\bf Case~1:} The necklace $R$ contains only the letters with values $k$ and $k+1$. Without loss of generality, we assume that $k=1$. First replace all $1$'s with $2$'s and all $2$'s with $1$'s, leaving the bars and colors in their original positions. Now the problem is that each $0$-colored $1$ that is followed by a $2$ has a bar and each $0$-colored $2$ that is followed by by a $1$ lacks a bar. We call a $1$ that is followed by a $2$ a rising $1$ and a $2$ that is followed by a $1$ a falling $2$. Since the number of rising $1$ equals the number of falling $2$ and they appear alternately, we can switch the color of each rising $1$ with the color of its followed falling $2$ and if in addition, the rising $1$ is $0$-colored with a bar then we also move the bar to its followed falling $2$, thereby obtaining a colored necklace $R'$ with the same number of bars and the same color vector as $R$ but with the number of $1$'s and $2$'s exchanged. Let $\psi(R)=R'$. Clearly, $\psi$ is reversible. For example if $R=(2^2\,\overline{2^0}\, 1^1\,\overline{1^0}\,1^0\,\overline{2^0}\,2^3\,\overline{2^0}\,2^1\,1^0\,1^0\,\overline{2^0}\,1^2\,\overline{1^0}\,1^0)$ then we  get  
$(1^2\,\overline{1^0}\, 2^1\,\overline{2^0}\,2^0\,\overline{1^0}\,1^3\,\overline{1^0}\,1^1\,2^0\,2^0\,\overline{1^0}\,2^2\,\overline{2^0}\,2^0)$ 
before the colors and bars are adjusted. After the colors and bars are adjusted we have 
$\psi(R)=(1^2\,1^0\, 2^1\,\overline{2^0}\,\overline{2^0}\,\overline{1^0}\,1^3\,\overline{1^0}\,1^0\,2^0\,2^1\,1^0\,2^2\,\overline{2^0}\,\overline{2^0})$.

\vskip 0.1in
{\bf Case 2:} The necklace $R$ has letters with values $k$ and $k+1$, and other letters which we will call intruders. The intruders enable us to form linear segments of $R$ consisting only of  $k$ 's and $(k+1)$'s. To obtain such a linear segment start with a letter of value $k$ or $k+1$ that follows an intruder and read the letters of $R$ in a clockwise direction until another intruder is encountered. For example if 
\begin{equation}\label{exmR}
R=(\overline{5^0}\,3^1\,3^0\,4^2\,\overline{4^0}\,\overline{3^0}\,3^1\,\overline{3^0}\,3^2\,6^2\,\overline{6^0}\,\overline{3^0}\,3^0\,3^1\,\overline{4^0}\,2^0\,4^3\,4^0)
\end{equation}
and $k=3$ then the segments are $3^1\,3^0\,4^2\,\overline{4^0}\,\overline{3^0}\,3^1\,\overline{3^0}\,3^2$, $\overline{3^0}\,3^0\,3^1\,\overline{4^0}$ and $4^3\,4^0$.

There are two types of segments, even segments and odd segments. An even (odd) segment contains an even (odd) number of switches, where a switch is a letter of value $k$ followed by one of value $k+1$ (call a rising $k$) or a letter of value $k+1$ followed by one of value $k$ (call a falling $k+1$). We handle the even and odd segments differently.
\vskip 0.1in
\noindent {\bf Subcase 1:} Even segments. In an even segment, we replace all $k$'s
 with $(k+1)$'s and all $(k+1)$'s with $k$'s. Again, this may product problems on rising $k$ or falling $k+1$. So we switch the color of  $i$-th rising $k$ with the color of $i$-th falling $k+1$ and move the bar (if it really has) from $i$-th rising $k$ to $i$-th falling $k+1$ to obtain a good segment, where we count rising $k$'s and falling $(k+1)$'s from left to right. This preserves the number of bars and color vector and exchanges the number of $k$'s and $(k+1)$'s. For example, the even segment 
 $3^1\,3^0\,4^2\,\overline{4^0}\,\overline{3^0}\,3^1\,\overline{3^0}\,3^2$ 
 gets replaced by 
 $4^1\,4^0\,3^2\,\overline{3^0}\,\overline{4^0}\,4^1\,\overline{4^0}\,4^2$. 
 After the bars and colors are adjusted we obtain
 $4^1\,\overline{4^0}\,3^2\,3^0\,\overline{4^0}\,4^1\,\overline{4^0}\,4^2$.
\vskip 0.1in
\noindent {\bf Subcase 2:} Odd segments. An odd segment either either starts with a $k$ and ends with a $k+1$ or vice versa. Both cases are handled similarly. So we suppose we have an odd segment of the form
$$
k^{m_1}(k+1)^{n_1}k^{m_2}(k+1)^{n_2}\cdots k^{m_r}(k+1)^{n_r},
$$
where each $m_i,n_i>0$ and the bars and colors have been suppressed. The number of switches is $2r-1$. We replace it with the odd segment
$$
k^{n_1}(k+1)^{m_1}k^{n_2}(k+1)^{m_2}\cdots k^{n_r}(k+1)^{m_r},
$$
and put bars and colors in their original positions. Again we may have created problems on rising $k$'s (but not on falling $(k+1)$'s); so we need to adjust bars and colors around. Note that the positions of the rising $k$'s are in the set $\{N_1+n_1,N_2+n_2,N_3+n_3,\ldots,N_{r}+n_r\}$, where $N_i=\sum_{t=1}^{i-1}(n_t+m_t)$. Now we switch the color  in position $N_i+n_i$ with the color in position $N_i+m_i$ and move the bar (if it really has) to position $N_i+m_i$, thereby obtain a good segment. For example, the odd segment $\overline{3^0}\,3^0\,3^1\,\overline{4^0}$ gets replaced by $\overline{3^0}\,4^0\,4^1\,\overline{4^0}$ before the bars and colors are adjusted. After the bars and colors are adjusted we have $3^1\,4^0\,\overline{4^0}\,\overline{4^0}$.

\vskip 0.1in

Let $\psi(R)$ be the colored necklace obtained by replacing all the segments in the way described above.
For example if $R$ is the colored necklace given in~\eqref{exmR} then 
$$
\psi(R)=(\overline{5^0}\,4^1\,\overline{4^0}\,3^2\,3^0\,\overline{4^0}\,4^1\,\overline{4^0}\,4^2\,6^2\,\overline{6^0}\,3^1\,4^0\,\overline{4^0}\,\overline{4^0}\,2^0\,3^3\,3^0).
$$

\vskip 0.1in
It is easy to see that $\psi$ is reversible in all cases and thus is a bijection of colored necklaces. This  completes the proof of the theorem. 
\end{proof}

\subsection{Colored banners}
\label{col:banners}
We shall give a brief review of the colored banner interpretation of $Q_{\check{\lambda},k}$ introduced by Hyatt~\cite{hy} and stated his generating function formula for $Q_{n,k,\Vec{\alpha},\Vec{\beta}}$.  We also give a  slightly different colored banner interpretation of $Q_{\check{\lambda},k}$ that will be used next.

\begin{definition}[Definition~4.2 of~\cite{hy}]
 A {\em colored banner} is  a word $B$ over the alphabet $\mathcal{B}$ such that 
\begin{itemize}
\item[(1)] if $B(i)$ is barred then $|B(i)|\geq|B(i+1)|$,
\item[(2)] if $B(i)$ is 0-colored and unbarred, then $|B(i)|\leq|B(i+1)|$ or $i$ equals the length of $B$,
\item[(3)] the last letter of $B$ is unbarred.
\end{itemize}
\end{definition}

Recall that a {\em Lyndon word} over an ordered alphabet is a word that is strictly lexicographically larger than all its circular rearrangements. It is a result of Lyndon  (cf. \cite[Theorem 5.1.5]{lo}) that every word has a unique  factorization into a lexicographically weakly increasing sequence of Lyndon words, called \emph{Lyndon factorization}. We say that a word of length $n$ has Lyndon type $\lambda$ (where $\lambda$ is a partition of $n$) if parts of $\lambda$ equal the lengths of the words in the Lyndon factorization.

We apply Lyndon factorization to colored banners.  The {\em cv-cycle type} of a colored banner B  is the multiset 
\begin{equation*}
\check{\lambda}(B)=\left\{(\lambda_1,\vec{\beta^1}),...,(\lambda_k,\vec{\beta^k})\right\}
\end{equation*}
if $B$ has Lyndon type $\lambda$, and the corresponding word of length $\lambda_i$ in the Lyndon factorization has color vector $\vec{\beta^i}$. The {\em weight} wt$(B)$ of a banner is defined to be the product of the weights of all letters in $B$.

\begin{thm}[New colored banner interpretation]\label{new:banner}
For all $\check{\lambda}$ and $k$,
$$
Q_{\check{\lambda},k}=\sum_{B}
\wt(B)
$$ 
summed all banners $B$ of length $n$ and cv-cycle type $\check{\lambda}$ (with respect to the order in~\eqref{original:order}) with exactly $k$ barred letters.
\end{thm}
\begin{proof}
The proof applies Lyndon factorization with respect to the order of $\mathcal{B}$ 
in~\eqref{original:order} to the banners  and is identical to the proof of~\cite[Theorem~3.6]{sw}.
\end{proof}

\begin{remark} \label{hyatt:ban}
Consider another order $<_B$ on the alphabet $\mathcal{B}$ as follows
\[1^1<_B \cdots <_B 1^{l-1}<_B 2^1<_B \cdots <_B 2^{l-1}<_B\cdots<_B n^1<_B \cdots <_B n^{l-1}<_B\]
\[<_B1^0<_B\overline{1^0}<_B 2^0<_B\overline{2^0}<_B 3^0<_B\overline{3^0}<_B \cdots n^0<_B\overline{n^0}.\]
Hyatt~\cite[Theorem~4.3]{hy} applied the Lyndon factorization to the colored banners with the above order $<_B$ on $\mathcal{B}$ to give a different colored banner interpretation of $Q_{\check{\lambda},k}$, which we should call the {\em original colored banner interpretation}. Our  new colored banner interpretation stated here is closer to the word interpretation in Lemma~\ref{word:version}, while the original colored banner interpretation will be used in the proof of Theorem~\ref{decomp:refine}.
\end{remark}

The following generating function for $Q_{n,k,\Vec{\alpha},\Vec{\beta}}$ was computed in~\cite{hy} by establishing a recurrence formula based on the original colored banner interpretation of $Q_{\check{\lambda},k}$. 

\begin{thm}[Hyatt] \label{hyatt}
Fix $l\in\P$ and let $r^{\Vec{\alpha}}=r_0^{\alpha_0}\cdots r_{l-1}^{\alpha_{l-1}}$
and $s^{\Vec{\beta}}=s_1^{\beta_1}\cdots s_{l-1}^{\beta_{l-1}}$. Then 
\begin{equation}\label{quasi-B}
\sum_{n,k\geq0\atop{\Vec{\alpha}\in\N^l,\Vec{\beta}\in\N^{l-1}}} Q_{n,k,\Vec{\alpha},\Vec{\beta}}z^nt^kr^{\Vec{\alpha}}s^{\Vec{\beta}}=\frac{H(r_0z)(1-t)(\prod\limits_{m=1}^{l-1}E(-s_mz)H(r_ms_mz))}{(1+\sum\limits_{m=1}^{l-1}s_m)H(tz)-(t+\sum\limits_{m=1}^{l-1}s_m)H(z)},
\end{equation}
where $H(z):=\sum_{i\geq0}h_iz^i$ and $E(z):=\sum_{i\geq0}e_iz^i$.
\end{thm}

\section{The decrease value theorem with an application}

\label{dec-val-thm}

The main objective of this section is to show how~\eqref{quasi-B} can be deduced from the decrease value theorem directly.

\subsection{Decrease values in words} We now introduce some word statistics studied in~\cite{fh2,fh4}. Let $w=w_1w_2\cdots w_n$ be an arbitrary word over $\N$. Recall that an integer $i\in[n-1]$ is said to be a {\em descent} of $w$ if $w_i>w_{i+1}$; it is a {\em decrease} of $w$ if $w_i=w_{i+1}=\cdots=w_j>w_{j+1}$ for some $j$ such that $i\leq j\leq n-1$. The letter $w_i$ is said to be a \emph{decrease value} of $w$. The set of all decreases (resp.~descents) of $w$ is denoted by $\Dec(w)$ (resp.~$\Des(w)$). Each descent is a decrease, but not conversely. Hence $\Des(w)\subset\Dec(w)$. 

In parallel with the notions of descent and decrease, an integer $i\in[n]$ is said to be a {\em rise} of $w$ if $w_i<w_{i+1}$ (By convention that $w_{n+1}=\infty$, and thus $n$ is always a rise); it is a {\em increase} of $w$ if $i\notin\Dec(w)$. The letter $w_i$ is said to be a \emph{increase value} of $w$.  The set of all increases (resp.~rises) of $w$ is denoted by $\Inc(w)$ (resp.~$\Rise(w)$). Clearly, each rise is a increase, but not conversely. Hence $\Rise(w)\subset\Inc(w)$.

Furthermore, a position $i$ is said to be a {\em record} if $w_i\geq w_j$ for all $j$ such that $1\leq j\leq i-1$ and the letter $w_i$ is called a {\em record value}. Denote by $\Rec(w)$ the set of all records of $w$.

Now, we define a mapping $f$ from words on $\N$ to colored banners as follows
$$
f: w=w_1w_2\ldots w_n\mapsto B=B(1)B(2)\ldots B(n),
$$
where
\begin{itemize}
\item $B(i)=\overline{u^0}$, if $w_i$ is a decrease value such that $w_i=ul$ for some $u\in\P$;
\item otherwise $B(i)=(u+1)^m$, where $w_i=ul+m$ for some $u,m\in\N$ satisfies $0\leq m\leq l-1$ and either $w_i$ is an increase or $m\neq0$.
\end{itemize}
For example, if $l=3$, then $f(12\,10\,9\,12\,8\,12\,16\,2\,13\,19)=\overline{4^0}\,4^1\,4^0\,\overline{4^0}\,3^2\,5^0\,6^1\,1^2\,5^1\,7^1$.

We should check that  such a word   $B$ over $\mathcal{B}$ is a colored banner. In the definition of a colored banner, condition $(3)$ is satisfied since the last letter of a word is always a increase value. 
If $B(i)$ is barred, then $w_i$ is a decrease value and so $w_i\geq w_{i+1}$, which would lead $|B(i)|\geq |B(i+1)|$, and thus condition $(1)$ is satisfied. Similarly, condition $(2)$ is also satisfied. This shows that $f$ is  well defined. 

A letter  $k\in\N$ is called a $m$-colored letter (or value) if  it  is congruent to $m$ modulo $l$. For a word $w=w_1\ldots w_n$ over $\N$, we define the {\em colored vector} $\vec{\col}(w)\in\N^{l-1}$ of $w$ to be 
$$
\vec{\col}(w):=(\col_1(w),\ldots,\col_{l-1}(w)),
$$
where $\col_m(w)$ is the number of $m$-colored letters in $w$  for $m=1,\ldots, l-1$.
Supposing that $w_i=u_il+m_i$ for some $0\leq m_i\leq l-1$, we then define the weight wt$(w)$ of $w$ to be the monomial $x_{d(w_1)}\ldots x_{d(w_n)}$, where 
  $d(w_i)=u_i$ if $w_i$ is a decrease value and $m_i=0$, otherwise $d(w_i)=u_i+1$. 
We also define the {\em cv-cycle type} of $w$ to be the multiset 
\begin{equation*}
\check{\lambda}(w)=\left\{(\lambda_1,\vec{\alpha^1}),...,(\lambda_k,\vec{\alpha^k})\right\}
\end{equation*}
if $w$ has Lyndon type $\lambda$ (with respect to the order of $\P$), and the corresponding word of length $\lambda_i$ in the Lyndon factorization has color vector $\vec{\alpha^i}$.

\begin{lem}\label{word:version}
 Let $W(\check{\lambda},k)$ be the set of all words over $\N$ with length $n$ and cv-cycle type $\check{\lambda}$ with exactly $k$ $0$-colored decrease values. Then 
 $$
 Q_{\check{\lambda},k}=\sum_{w\in W(\check{\lambda},k)}
\wt(w).
 $$
 \end{lem}
 \begin{proof}
Clearly, the mapping $f$ is a bijection which maps $0$-colored decrease values  to $0$-colored barred  letters and preserves the color of letters. It is also weight preserving $\wt(w)=\wt(f(w))$. Recall the order of $\mathcal{B}$ in~\eqref{original:order}.
It is not hard to check  that if the Lyndon factorization of a word $w$ over $\N$ is 
$$w=(w_1)(w_2)\cdots(w_k),$$ then the Lyndon factorization (with respect to the above order of $\mathcal{B}$) of the banner $f(w)$ is 
$$f(w)=(f(w_1))(f(w_2))\cdots(f(w_k)).$$ Thus $f$ also keeps the Lyndon factorization type, which would complete the proof in view of Theorem~\ref{new:banner}.
\end{proof}

\subsection{Combinatorics of the decrease value theorem}
Let $[0,r]^*$ be the set of all finite words whose letters are taken from the alphabet $[0,r]=\{0,1,\ldots,r\}$. Introduce six sequences of commuting variables $(X_i),(Y_i),(Z_i),(T_i),(Y_i'),(T_i')$ ($i=0,1,2,\ldots$), and for each word $w=w_1w_2\ldots w_n$ from $[0,r]^*$ define the {\em weight} $\psi(w)$ of $w$ to be
\begin{align}
\psi(w):=&\prod_{i\in\Des}X_{w_i}\prod_{i\in\Rise\setminus\Rec}Y_{w_i}\prod_{i\in\Dec\setminus\Des}Z_{w_i}\\
&\times\prod_{i\in(\Inc\setminus\Rise)\setminus\Rec}T_{w_i}\prod_{i\in\Rise\cap\Rec}Y_{w_i}'\prod_{i\in(\Inc\setminus\Rise)\cap\Rec}T_{w_i}'.\nonumber
\end{align}

The following  generating function for the set $[0,r]^*$ by the weight $\psi$ was calculated by Foata and Han~\cite{fh4} using the properties of {\em Foata's first fundamental transformation} on words (see~\cite[Chap.~10]{lo}) and a noncommutative version of {\em MacMahon Master Theorem} (see~\cite[Chap.~4]{cf}).  

\begin{thm}[Decrease value theorem] We have:
\begin{equation}\label{decrease}
\sum_{w\in[0,r]^*}\psi(w)=\frac{\frac{\prod\limits_{1\leq j\leq r}\frac{1-Z_j}{1-Z_j+X_j}}{\prod\limits_{0\leq j\leq r}\frac{1-T_j'}{1-T_j'+Y_j'}}}{1-\sum\limits_{1\leq k\leq r}\frac{\prod\limits_{1\leq j\leq k-1}\frac{1-Z_j}{1-Z_j+X_j}}{\prod\limits_{0\leq j\leq k-1}\frac{1-T_j}{1-T_j+Y_j}}\frac{X_k}{1-Z_k+X_k}}.
\end{equation}
\end{thm}

 We show in the following that one can also use the  {\em Kim-Zeng decomposition of multiderangement}~\cite{kz} (but not the word-analog of the Kim-Zeng decomposition developed in~\cite[Theorem~3.4]{fh2}) instead of MacMahon Master Theorem to prove the decrease value theorem combinatorially. 

A letter $w_i$ which is a record and also a rise value is called a {\em riserec value}. A word $w\in[0,r]^*$ having no equal letters in succession is called {\em horizontal derangement}.  Denote by $[0,r]_d^*$ the set of all the horizontal derangement words  in $[0,r]^*$ without riserec value.
It was shown in~\cite{fh4}  that the decrease value theorem is equivalent to 
\begin{align*}
\sum_{w\in[0,r]_d^*}\psi(w)
=\frac{1}{\prod\limits_{1\leq j\leq r}(1+X_j)-\sum\limits_{1\leq i\leq r}\left(\prod\limits_{0\leq j\leq i-1}(1+Y_j)\prod\limits_{i+1\leq j\leq r}(1+X_j)\right)X_i},
\end{align*}
which again can be rewritten as
\begin{align}\label{decrease:spec}
\sum_{w\in[0,r]_d^*}\psi(w)
=\frac{1}{1-\sum\limits_{1\leq i\leq r}\left(\left(\prod\limits_{0\leq j\leq i-1}(1+Y_j)-1\right)\prod\limits_{i+1\leq j\leq r}(1+X_j)\right)X_i}.
\end{align}
Using Foata's first fundamental transformation on words, we can factorize each word in $[0,r]_d^*$ as a product of cycles of length at least $2$, where  the rises of the word are transform to the  excedances of the cycles. Recall that a cycle $\sigma=s_1s_2\cdots s_k$ is called a {\em prime cycle} if there exists $i$, $2\leq i\leq k$, such that $s_1<\cdots<s_{i-1}<s_k<s_{k-1}<\cdots<s_i$.
By the two decompositions in~\cite{kz}, every cycle of length at least $2$ admits a decomposition to some components of prime cycles, from which we can see Eq.~\eqref{decrease:spec} directly. 


\subsection{A new proof of Hyatt's result}

Introduce three sequences of commuting variables $(\xi_i), (\eta_i), (\zeta_i),\,(i=0,1,2,\ldots)$
and make the  following substitutions:
$$
X_i=\xi_i,\quad Z_i=\xi_i,\quad Y_i=\eta_i,\quad T_i=\eta_i, \quad Y_i'=\zeta_i,\quad T_i'=\zeta_i\quad (i=0,1,2,\ldots).
$$
The new weight $\psi'(w)$ attached to each word $w=y_1y_2\cdots y_n$ is then
\begin{equation}\label{weight}
\psi'(w)=\prod_{i\in\Dec(w)}\xi_{y_i}\prod_{i\in(\Inc \setminus \Rec)(w)}\eta_{y_i}\prod_{i\in(\Inc \cap \Rec)(w)}\zeta_{y_i},
\end{equation}
and identity~\eqref{decrease} becomes:
\begin{equation}\label{decrease2}
\sum_{w\in[0,r]^*}\psi'(w)=\frac{\frac{\prod\limits_{1\leq j\leq r}(1-\xi_j)}{\prod\limits_{0\leq j\leq r}(1-\zeta_j)}}{1-\sum\limits_{1\leq k\leq r}\frac{\prod\limits_{1\leq j\leq k-1}(1-\xi_j)}{\prod\limits_{0\leq j\leq r}(1-\eta_j)}\xi_k}.
\end{equation}

Let $\eta$ denote the homomorphism defined by the following substitutions of variables:
\begin{align*}
\eta:=
\begin{cases}
\xi_j\leftarrow tY_{i-1}, \zeta_j\leftarrow r_0Y_{i}, \eta_j\leftarrow Y_{i},&\quad\text{if $j=li$}; \\
\xi_j\leftarrow s_mY_{i}, \zeta_j\leftarrow r_ms_mY_{i}, \eta_j\leftarrow s_mY_{i},&\quad\text{if $j=li+m$ for some $1\leq m\leq l-1$}.
\end{cases}
\end{align*}

\begin{lem}\label{le2} We have
\begin{align*}
\frac{\prod_{j\geq0}(1-sY_j)-\prod_{j\geq0}(1-Y_j)}{\prod_{j\geq0}(1-Y_j)}=(1-s)\sum_{i\geq0}Y_i\frac{\prod_{0\leq j\leq i-1}(1-sY_j)}{\prod_{0\leq j\leq i}(1-Y_j)}. 
\end{align*}
\end{lem}
\begin{proof}
First, we may check that 
\begin{align*}
&\prod_{0\leq j\leq r}(1-sY_j)-\prod_{0\leq j\leq r}(1-Y_j)\\
=&\sum_{0\leq i\leq r}\prod_{0\leq j\leq i}(1-sY_j)\prod_{i+1\leq j\leq r}(1-Y_j)-\sum_{0\leq i\leq r}\prod_{0\leq j\leq i-1}(1-sY_j)\prod_{i\leq j\leq r}(1-Y_j)\\
=&(1-s)\sum_{0\leq i\leq r}Y_i\prod_{0\leq j\leq i-1}(1-sY_j)\prod_{i+1\leq j\leq r}(1-Y_j).
\end{align*}
Multiplying both sides by $\frac{1}{\prod_{0\leq j\leq r}(1-Y_j)}$ yields 
\begin{align*}
\frac{\prod_{0\leq j\leq r}(1-sY_j)-\prod_{0\leq j\leq r}(1-Y_j)}{\prod_{0\leq j\leq r}(1-Y_j)}=(1-s)\sum_{0\leq i\leq r}Y_i\frac{\prod_{0\leq j\leq i-1}(1-sY_j)}{\prod_{0\leq j\leq i}(1-Y_j)}. 
\end{align*}
Letting $r$ tends to infinity, we get the desired formula.
\end{proof}

\begin{thm}We have
\begin{align}\label{maineq}
\lim_{r\rightarrow\infty}\sum_{w\in[0,r]^*}\eta\psi'(w)=\frac{H(r_0Y)(1-t)(\prod_{m=1}^{l-1}E(-s_mY)H(r_ms_mY))}{(1+\sum_{m=1}^{l-1})H(tY)-(t+\sum_{m=1}^{l-1})H(Y)},
\end{align}
where $H(tY)=\prod_{i\geq0}(1-tY_i)^{-1}$ and $E(sY)=\prod_{i\geq0}(1+sY_i)$.
\end{thm}

\begin{proof}
By~\eqref{decrease2}, we have
\begin{align*}
\sum_{w\in[0,r]^*}\eta\psi'(w)=&\frac{\frac{\prod_{1\leq i\leq \lfloor\frac{r}{l}\rfloor}(1-tY_{i-1})\prod_{m=1}^{l-1}\left(\prod_{0\leq i\leq \lfloor\frac{r-m}{l}\rfloor}(1-s_mY_i)\right)}{\prod_{0\leq i\leq \lfloor\frac{r}{l}\rfloor}(1-r_0Y_{i-1})\prod_{m=1}^{l-1}\left(\prod_{0\leq i\leq \lfloor\frac{r-m}{l}\rfloor}(1-r_ms_mY_i)\right)}}{1-\sum_{1\leq k\leq r}\frac{\prod_{1\leq i\leq\lfloor\frac{k-1}{l}\rfloor}(1-tY_{i-1})\prod_{m=1}^{l-1}\left(\prod_{0\leq i\leq \lfloor\frac{k-1-m}{l}\rfloor}(1-s_mY_i)\right)}{\prod_{0\leq i\leq\lfloor\frac{k-1}{l}\rfloor}(1-Y_{i})\prod_{m=1}^{l-1}\left(\prod_{0\leq i\leq \lfloor\frac{k-1-m}{l}\rfloor}(1-s_mY_i)\right)}\eta(\xi_k)}\\
=&\frac{\frac{\prod_{1\leq i\leq \lfloor\frac{r}{l}\rfloor}(1-tY_{i-1})\prod_{m=1}^{l-1}\left(\prod_{0\leq i\leq \lfloor\frac{r-m}{l}\rfloor}(1-s_mY_i)\right)}{\prod_{0\leq i\leq \lfloor\frac{r}{l}\rfloor}(1-r_0Y_{i-1})\prod_{m=1}^{l-1}\left(\prod_{0\leq i\leq \lfloor\frac{r-m}{l}\rfloor}(1-r_ms_mY_i)\right)}}{1-\sum_{1\leq k\leq r}\frac{\prod_{1\leq i\leq\lfloor\frac{k-1}{l}\rfloor}(1-tY_{i-1})}{\prod_{0\leq i\leq\lfloor\frac{k-1}{l}\rfloor}(1-Y_{i})}\eta(\xi_k)}.
\end{align*}

Thus, we obtain
\begin{align}\label{formula1}
\lim_{r\rightarrow\infty}\sum_{w\in[0,r]^*}\eta\psi'(w)=\frac{\frac{\prod_{i\geq 0}(1-tY_{i})\prod_{m=1}^{l-1}\prod_{i\geq 0}(1-s_mY_i)}{\prod_{i\geq 0}(1-r_0Y_{i-1})\prod_{m=1}^{l-1}\prod_{i\geq 0}(1-r_ms_mY_i)}}{1-\sum_{k\geq1}\frac{\prod_{1\leq i\leq\lfloor\frac{k-1}{l}\rfloor}(1-tY_{i-1})}{\prod_{0\leq i\leq\lfloor\frac{k-1}{l}\rfloor}(1-Y_{i})}\eta(\xi_k)}.
\end{align}
By the definition of $\eta$,
\begin{align*}
&1-\sum_{k\geq1}\frac{\prod_{1\leq i\leq\lfloor\frac{k-1}{l}\rfloor}(1-tY_{i-1})}{\prod_{0\leq i\leq\lfloor\frac{k-1}{l}\rfloor}(1-Y_{i})}\eta(\xi_k)\nonumber\\
=&1-\prod_{i\geq0}\frac{\prod_{0\leq j\leq i-1}(1-tY_j)}{\prod_{0\leq j\leq i}(1-tY_j)}tY_i-\sum_{m=1}^{l-1}\left(\prod_{i\geq0}\frac{\prod_{0\leq j\leq i-1}(1-tY_j)}{\prod_{0\leq j\leq i}(1-tY_j)}s_mY_i\right)\\
=&1-(t+\sum_{m=1}^{l-1}s_m)\prod_{i\geq0}\frac{\prod_{0\leq j\leq i-1}(1-tY_j)}{\prod_{0\leq j\leq i}(1-tY_j)}Y_i.
\end{align*}
By Lemma~\ref{le2}, the above identity becomes
\begin{align*}
&1-\sum_{k\geq1}\frac{\prod_{1\leq i\leq\lfloor\frac{k-1}{l}\rfloor}(1-tY_{i-1})}{\prod_{0\leq i\leq\lfloor\frac{k-1}{l}\rfloor}(1-Y_{i})}\eta(\xi_k)\\
=&1-\left(\frac{t+\sum_{m=1}^{l-1}s_m}{1-t}\right)\left(\frac{\prod_{j\geq0}(1-tY_j)-\prod_{j\geq0}(1-Y_j)}{\prod_{j\geq0}(1-Y_j)}\right).
\end{align*}
After substituting this expression into~\eqref{formula1}, we get~\eqref{maineq}
\end{proof}

Combining the above theorem with Lemma~\ref{word:version} we get a decrease value theorem approach to Hyatt's generating function~\eqref{quasi-B}.

%


%

\section{Flag Eulerian quasisymmetric functions}
\label{flag:euler} 

Let $\cs(\Vec{\beta}):=\sum_{i=1}^{l-1}i\times\beta_i$ for each $\Vec{\beta}=(\beta_1,\ldots, \beta_{l-1})\in \N^{l-1}$. 
We define the {\em Flag Eulerian quasisymmetric functions} $Q_{n,k,j}$  as
$$
Q_{n,k,j}:=\sum_{i,\Vec{\alpha},\Vec{\beta}} Q_{n,i,\Vec{\alpha},\Vec{\beta}},
$$
where the sum is over all $i,\Vec{\alpha}\in\N^l,\Vec{\beta}\in\N^{l-1}$ such that $li+\cs(\Vec{\beta})=k$ and $\alpha_0=j$.

\begin{cor}[of Theorem~\ref{hyatt}] \label{hy:flagexc}
We have 
\begin{equation}\label{eq:flagexc}
\sum_{n,k,j\geq0} Q_{n,k,j}t^kr^jz^n=\frac{(1-t)H(rz)}{H(t^lz)-tH(z)},
\end{equation}
where $Q_{0,0,0}=1$.
\end{cor}

For a positive integer, the polynomial $[n]_q$ is defined as 
$$
[n]_q:=1+q+\cdots+q^{n-1}.
$$
By convention, $[0]_q=0$. 
\begin{cor}Let $Q_n(t,r)=\sum_{j,k\geq0} Q_{n,k,j}t^kr^j$. Then $Q_n(r,t)$ satisfies the following recurrence relation:
\begin{equation}\label{rec:wachs}
Q_n(t,r)=r^nh_n+\sum_{k=0}^{n-1}Q_k(t,r)h_{n-k}t[l(n-k)-1]_t.
\end{equation}
Moreover,
\begin{equation}\label{exa:wachs}
Q_n(t,r)=\sum_{m}\sum_{k_0\geq0\atop{lk_1,\ldots,lk_m\geq2\atop\sum_{k_i}=n}}r^{k_0}h_{k_0}\prod_{i=1}^mh_{k_i}t[lk_i-1]_t.
\end{equation}
\end{cor}
\begin{proof} By~\eqref{eq:flagexc}, we have
\begin{equation*}
\sum_{n,k,j\geq0} Q_n(t,r)z^n=\frac{H(rz)}{1-\sum_{n\geq1}t[ln-1]_th_nz^n},
\end{equation*}
which is equivalent to~\eqref{rec:wachs}. It is not hard to show that the right-hand 
side of~\eqref{exa:wachs} satisfies the recurrence relation~\eqref{rec:wachs}. This proves~\eqref{exa:wachs}.
\end{proof}

Bagno and Garber~\cite{bg} introduced the \emph{flag excedance} statistic for each colored permutation $\pi\in C_l\wr\S_n$, denoted by $\fexc(\pi)$, as 
$$
 \fexc(\pi):=l\cdot\exc(\pi)+\sum_{i=1}^n\epsilon_i.
$$
Note that  when $l=1$, flag excedances are excedances on permutations. Define the number of fixed points of $\pi$, $\fix(\pi)$, by 
$$\fix(\pi):=\fix_0(\pi).$$
Now we define the {\em colored $(q,r)$-Eulerian polynomials} $A_n^{(l)}(t,r,q)$ by 
$$
A_n^{(l)}(t,r,q):=\sum_{\pi\in C_l\wr\S_n}t^{\fexc(\pi)}r^{\fix(\pi)}q^{(\maj-\exc)\pi}.
$$
Let $A_n^{(l)}(t,q):=A_n^{(l)}(t,1,q)$.
Then by~\eqref{colored:eulnum} and~\eqref{coloredn:eulnum},
$$
A_n^{(l)}(t,r,q)=\sum_{k,j}A_{n,k,j}^{(l)}(q)t^kr^j\quad\text{and}\quad A_n^{(l)}(t,q)=\sum_{k}A_{n,k}^{(l)}(q)t^k.
$$
The following specialization follows immediately from Lemma~\ref{DEX:lem}.

\begin{lem}\label{lem:psflag}
Let $Q_n(t,r)=\sum_{j,k\geq0} Q_{n,k,j}t^kr^j$. Then we have
$$
\ps(Q_n(t,r))=(q;q)_n^{-1}A_n^{(l)}(t,r,q).
$$
\end{lem}

Let the {\em$q$-multinomial coefficient} ${\bmatrix n\\ k_0,\ldots,k_m\endbmatrix}_q$ be
$$
{\bmatrix n\\ k_0,\ldots,k_m\endbmatrix}_q=\frac{(q;q)_n}{(q;q)_{k_0}\cdots(q;q)_{k_m}}.
$$
Applying the specialization to both sides of~\eqref{eq:flagexc},~\eqref{rec:wachs} and~\eqref{exa:wachs} yields the following formulas for $A_n^{(l)}(t,r,q)$.

\begin{cor}
We have
\begin{equation}  \label{expo-color}
\sum_{n\geq0}A_n^{(l)}(t,r,q)\frac{z^n}{(q;q)_n}
=\frac{(1-t)e(rz;q)}{e(t^lz;q)-te(z;q)}.
\end{equation}
\end{cor}

\begin{remark}
The above generalization of~\eqref{fixversion} can also be deduced from~\cite[Theorem~1.3]{fh3} through some calculations; see the proof of~\cite[Theorem~5.2]{fh3} for details.
\end{remark}

\begin{cor}We have
$$
A_n^{(l)}(t,r,q)=r^n+\sum_{k=0}^{n-1}{\bmatrix n\\ k\endbmatrix}_qA_k^{(l)}(t,r,q)t[l(n-k)-1]_t
$$
and
$$
A_n^{(l)}(t,r,q)=\sum_m\sum_{k_0\geq0\atop{lk_1,\ldots,lk_m\geq2\atop\sum_{k_i}=n}}{\bmatrix n\\ k_0,\ldots,k_m\endbmatrix}_qr^{k_0}\prod_{i=1}^m[lk_i-1]_t.
$$
\end{cor}

\subsection{Symmetry and unimodality}

Let $A(t)=a_rt^r+a_{r+1}t^{r+1}+\cdots+a_st^s$ be a nonzero polynomial in $t$ whose coefficients come from a partially ordered ring $R$. We say that $A(t)$ is {\em $t$-symmetric} (or symmetric when $t$ is understood) with center of symmetry $\frac{s+r}{2}$ if  $a_{r+k}=a_{s-k}$ for all $k=0,1,\ldots,s-r$ and also {\em$t$-unimodal} (or unimodal when $t$ is understood) if
$$
a_r\leq_R a_{r+1}\leq_R\cdots\leq_R a_{\lfloor\frac{s+r}{2}\rfloor}=a_{\lfloor\frac{s+r+1}{2}\rfloor}\leq\cdots\leq_R a_{s-1}\leq_R a_s.
$$
We also say that $A(t)$ is {\em log-concave} if $a_k^2\geq a_{k-1}a_{k+1}$ for  all $k=r+1,r+2,\ldots,s-1$. 

It is well known that a polynomial with positive coefficients and with only real roots is log-concave and that log-concavity implies unimodality. 
 For each fixed $n$, the Eulerian polynomial $A_n(t)=\sum_{k=0}^{n-1}A_{n,k}t^k$ is $t$-symmetric and has only real roots and therefore $t$-unimodal (see~\cite[p.~292]{co}). 
 
 The following fact is known~\cite[Proposition~1]{st3} and should not be difficult to prove. 
\begin{lem}\label{fact:unimodal}
The product of two symmetric unimodal polynomials with respective centers of symmetry $c_1$ and $c_2$ is symmetric and unimodal with center of symmetry $c_1+c_2$. 
\end{lem}
 
 Let the {\em colored Eulerian polynomials} $A_n^{(l)}(t)$ be defined as 
 $$
 A_n^{(l)}(t):=A_n^{(l)}(t,1,1)=\sum_{\pi\in C_l\wr\S_n}t^{\fexc(\pi)}.
 $$
 Recently, Mongelli~\cite[Proposition~3.3]{mon} showed that 
 $$
 A_n^{(2)}(t)=A_n(t)(1+t)^n,
 $$
which implies that $A_n^{(2)}(t)$ is $t$-symmetry and has only real roots and therefore $t$-unimodal. His idea can be extended to general $l$. Actually, we can construct $\pi$ by putting the colors to all entries of $|\pi|$. Analyzing how the concerned statistics are changed according to the entry what we put the color to is an excedance or nonexcedance of $|\pi|$ then gives   
\begin{align*}
&\sum_{\pi\in C_l\wr\S_n}t^{l\cdot\exc(\pi)}s_1^{\col_1(\pi)}\cdots s_{l-1}^{\col_{l-1}(\pi)}\\
&=\sum_{h=0}^{n-1}A_{n,h}(t^l+s_1+s_2+\cdots+s_{l-1})^h(1+s_1+s_2+\cdots+s_{l-1})^{n-h}\\
&=A_n\left(\frac{t^l+s_1+s_2+\cdots+s_{l-1}}{1+s_1+s_2+\cdots+s_{l-1}}\right)(1+s_1+s_2+\cdots+s_{l-1})^n.
\end{align*}
Setting $s_i=t^i$ in the above equation yields
\begin{equation}\label{CSP}
A_n^{(l)}(t)=A_n(t)(1+t+t^2+\cdots+t^{l-1})^n.
\end{equation}
From this and Lemma~\ref{fact:unimodal} we see that $A_n^{(l)}(t)$ is $t$-symmetric and $t$-unimodal with center of symmetry $\frac{ln-1}{2}$ although not real-rootedness when $l>2$. Note that relationship~\eqref{CSP} can also be deduced from~\eqref{expo-color} directly. It is known (cf.~\cite[Proposition~2]{st3}) that the product of two 
log-concave polynomials with positive coefficients is again log-concave, thus by~\eqref{CSP} we have the following result.
\begin{pro} 
The polynomial $A_n^{(l)}(t)$ is $t$-symmetric and log-concave for $l\geq1$. In particular it is $t$-unimodal. 
\end{pro}

 Let $d_n^B(t)$ be the generating function of the flag excedances on the derangements in $C_2\wr\S_n$, i.e.
$$
d_n^B(t):=\sum_{\pi\in C_2\wr\S_n\atop\fix(\pi)=0}t^{\fexc(\pi)}.
$$
Clearly, we have
$$d_n^B(t)=A_n^{(2)}(t,0,1).$$
At the end of~\cite{mon}, Mongelli noticed that $d_5^B(t)$ is not real-rootedness and conjectured that $d_n^B(t)$ is unimodal for any $n\geq1$.
 This conjecture motivates us to study the symmetry and unimodality of the coefficients of $t^k$ in the flag Eulerian quasisymmetric functions and the colored $(q,r)$-Eulerian polynomials. 
 
Let $\Par$ be the set of all partitions of all nonnegative integers. For a $\Q$-basis of the space of symmetric function $b=\{b_{\lambda} : \lambda\in\Par\}$, we have the partial order relation on the ring of symmetric functions given by 
$$
f\leq_b g\Leftrightarrow g-f\,\,\text{is $b$-positive},
$$
 where a symmetric function is said to be $b$-positive if it is a nonnegative linear combination of elements of the basis $\{b_{\lambda}\}$. Here we are concerned with the $h$-basis, $\{h_{\lambda} : \lambda\in\Par\}$ and the Schur basis $\{s_{\lambda} : \lambda\in\Par\}$. Since $h$-positivity implies Schur-positivity, the  following result also holds for the Schur basis.
 
The proof of the following theorem is similar to the proof of~\cite[Theorem~5.1]{sw}, which is the $l=1$ case of the following theorem.

\begin{thm} \label{hpositivity}
Let $Q_{n,k}=\sum_{j=0}^nQ_{n,k,j}$. Using the $h$-basis to partially order the ring of symmetric functions, we have for all $n,j,k$,
\begin{enumerate}
\item $Q_{n,k,j}$ is $h$-positive symmetric functions,
\item the polynomial $\sum_{k=0}^{ln-1}Q_{n,k,j}t^k$ is $t$-symmetric and $t$-unimodal with center of symmetry $\frac{l(n-j)}{2}$,
\item the polynomial $\sum_{k=0}^{ln-1}Q_{n,k}t^k$ is $t$-symmetric and $t$-unimodal with center of symmetry $\frac{ln-1}{2}$.
\end{enumerate}
\end{thm}

\begin{proof}
Part (1) follows from~\eqref{exa:wachs}.
We will use the fact in Lemma~\ref{fact:unimodal} to show Part (2) and (3).

By~\eqref{exa:wachs} we have 
$$
\sum_{k=0}^{ln-1}Q_{n,k,j}t^k=\sum_{m}\sum_{lk_1,\ldots,lk_m\geq2\atop\sum_{k_i}=n-j}h_j\prod_{i=1}^mh_{k_i}t[lk_i-1]_t.
$$
Each term $h_j\prod_{i=1}^mh_{k_i}t[lk_i-1]_t$ is $t$-symmetric and $t$-unimodal with center of symmetry $\sum_i\frac{lk_i}{2}=\frac{l(n-j)}{2}$. Hence the sum of these terms has the same property, which shows Part (2).

With a bit more effort we can show that Part (3) also follows from~\eqref{exa:wachs}.   For any sequence of positive integers $(k_1,\ldots,k_m)$, let 
$$
G_{k_1,\ldots,k_m}^{(l)}:=\prod_{i=1}^m h_{k_i}t[lk_i-1]_t.
$$
We have by~\eqref{exa:wachs}, 
$$
\sum_{k=0}^{ln-1}Q_{n,k,0}t^k=\sum_m\sum_{lk_1,\ldots,lk_m\geq2\atop\sum k_i=n} G_{k_1,\ldots,k_m}^{(l)}
$$
and
$$
\sum_{j\geq1}\sum_{k=0}^{ln-1}Q_{n,k,j}t^k=\sum_m\sum_{lk_1,\ldots,lk_m\geq2\atop\sum k_i=n} h_{k_1}G_{k_2,\ldots,k_m}^{(l)}
$$
assuming $l\geq2$. We claim that $G_{k_1,\ldots,k_m}^{(l)}+h_{k_1}G_{k_2,\ldots,k_m}^{(l)}$ is $t$-symmetric and $t$-unimodal with center of symmetry $\frac{ln-1}{2}$. Indeed, we have 
$$
G_{k_1,\ldots,k_m}^{(l)}+h_{k_1}G_{k_2,\ldots,k_m}^{(l)}=h_{k_1}(t[lk_1-1]_t+1)G_{k_2,\ldots,k_m}^{(l)}.
$$
Clearly $t[lk_1-1]_t+1=1+t+\cdots+t^{lk_1-1}$ is $t$-symmetric and $t$-unimodal with center of symmetry $\frac{lk_1-1}{2}$, and $G_{k_2,\ldots,k_m}$ is $t$-symmetric and $t$-unimodal with center of symmetry $\frac{l(n-k_1)}{2}$. Therefore our claim holds and the proof of Part (3) is complete because of
$$
\sum_{k=0}^{ln-1}Q_{n,k}t^k=\sum_{k=0}^{ln-1}Q_{n,k,0}t^k+\sum_{j\geq1}\sum_{k=0}^{ln-1}Q_{n,k,j}t^k.
$$
\end{proof}

\begin{remark}
We can give a bijective proof of the symmetric property
\begin{equation}\label{sym:flageul}
Q_{n,k,j}=Q_{n,l(n-j)-k,j}
\end{equation}
using the colored ornament interpretation of $Q_{n,k,j}$. 
We construct an involution $\varphi$ on colored ornaments such that 
if  the cv-cycle type of  a colored banner $R$ is
$$
\check{\lambda}(R)=\{(\lambda_1,\vec{\beta^1}),\ldots,(\lambda_r,\vec{\beta^r})\},
$$ then the cv-cycle type of $\varphi(R)$ is
$$
\check{\lambda}(\varphi(R))=\{(\lambda_1,\vec{\beta^1}^{\bot}),\ldots,(\lambda_r,\vec{\beta^r}^{\bot})\},
$$
where $\vec{\beta^{}}^{\bot}:=(\beta_{l-1},\beta_{l-2},\ldots,\beta_{1})$ for each $\vec{\beta}=(\beta_1,\ldots,\beta_{l-1})\in\N^{l-1}$. 
Let $R$ be a colored banner.
To obtain $\varphi(R)$, first we bar each unbarred $0$-colored letter of each nonsingleton colored necklace of $R$ and unbar each barred $0$-colored letter. Next we change the color of each $m$-colored  letter of $R$ to color $l-m$ for all $m=1,\ldots,l-1$. Finally for each $i$, we replace each occurrence of the $i$th smallest value in $R$ with the $i$th largest value leaving the bars and colors intact. 
 This involution shows~\eqref{sym:flageul} because $Q_{n,k,j}$ is a symmetric function by Theorem~\ref{symfun:cvcycle}.
\end{remark}

For the ring of polynomials $\Q[q]$, where $q$ is a indeterminate, we use the partial order relation:
$$
f(q)\leq_q g(q)\Leftrightarrow g(q)-f(q)\,\,\text{has nonnegative coefficients}.
$$
We will use the following simple fact from~\cite[Lemma~5.2]{sw}.

\begin{lem} \label{schur:ps}
If $f$ is a Schur positive homogeneous symmetric function of degree $n$ then $(q;q)_n\ps(f)$ is a polynomial in $q$ with nonnegative coefficients.
\end{lem}
\begin{thm}
For all $n,j$, 
\begin{enumerate}
\item The polynomial $\sum\limits_{\pi\in C_l\wr\S_n\atop\fix(\pi)=j}t^{\fexc(\pi)}q^{(\maj-\exc)\pi}$ is $t$-symmetric and $t$-unimodal with center of symmetry $\frac{l(n-j)}{2}$,
\item $A_n^{(l)}(t,q)$ is $t$-symmetric and $t$-unimodal with center of symmetry $\frac{ln-1}{2}$.
\end{enumerate}
\end{thm}
\begin{proof}
Since $h$-positivity implies Schur positivity, by Lemma~\ref{schur:ps}, we have that if $f$ and $g$ are homogeneous symmetric functions of degree $n$  and $f\leq_h g$ then 
$$
(q;q)_n\ps(f)\leq_q (q;q)_n\ps(g).
$$
By Lemma~\ref{lem:psflag},  Part (1) and (2) are obtained by specializing Part (2) and (3) of Theorem~\ref{hpositivity}, respectively.
\end{proof}

\begin{remark}
Part (1) of the above theorem implies the unimodality of $d_n^B(t)$ as conjectured in~[Conjecture~8.1]\cite{mon}. 
Actually, Mongelli~\cite[Conjecture~8.1]{mon} also conjectured that $d_n^B(t)$ is log-concave. Note that using the continued fractions, Zeng~\cite{zeng} found a symmetric and unimodal expansion of $d_n^B(t)$,  which also implies the unimodality of $d_n^B(t)$.
\end{remark}

 
%
%

\subsection{Generalized symmetrical Eulerian identities} 
In the following, we will give proofs of the two generalized symmetrical Eulerian identities in the introduction.
\vskip 0.3cm
\noindent\textbf{Theorem~\ref{sym:iden1}.}
For $a,b\geq1$ and $j\geq0$ such that $a+b+1=l(n-j)$, 
\begin{equation}\label{symsuns:sym}
\sum_{i\geq0}h_{i}Q_{n-i,a,j}=\sum_{i\geq0}h_iQ_{n-i,b,j}.
\end{equation}

\begin{proof}
Cross-multiplying and expanding all the functions $H(z)$ in~\eqref{eq:flagexc}, we obtain
\begin{align*}
\sum_{n\geq0}h_n(t^lz)^n\sum_{n,j,k\geq0}Q_{n,k,j}r^jt^kz^n-t\sum_{n\geq0}h_nz^n\sum_{n,j,k\geq0}Q_{n,k,j}r^jt^kz^n=(1-t)\sum_{n\geq0}h_n(rz)^n.
\end{align*}
Now, identifying the coefficients of $z^n$ yields
\begin{align*}
\sum_{i,j,k}h_iQ_{n-i,k-li,j}r^jt^k-\sum_{i,j,k}h_iQ_{n-i,k-1,j}r^jt^k=(1-t)h_nr^n.
\end{align*}
Hence, for $j<n$, we can identity the coefficients of $r^jt^k$ and obtain 
\begin{align}\label{sym:function}
\sum_{i\geq0}h_iQ_{n-i,k-li,j}=\sum_{i\geq0}h_iQ_{n-i,k-1,j}.
\end{align}
Applying the symmetry property~\eqref{sym:flageul} to the left side of the above equation yields
$$
\sum_{i\geq0}h_iQ_{n-i,l(n-j)-k,j}=\sum_{i\geq0}h_iQ_{n-i,k-1,j},
$$
which becomes~\eqref{symsuns:sym} after setting $a=k-1$ and $b=l(n-j)-k$ since now $n>j$.
\end{proof}

\begin{proof}[{\bf A bijective proof of Theorem~\ref{sym:iden1}}]
This bijective proof involves both the colored ornament  and the colored banner interpretations  of $Q_{n,k,j}$.

We will give a bijective proof of
\begin{equation}\label{sym:flagwith}
Q_{n,k}=Q_{n,ln-k-1}
\end{equation}
by means of colored banners, using Theorem~\ref{new:banner}. We describe an involution $\theta$ on colored banners. Let $B$ be a colored banner. To obtain $\theta(B)$, first we bar each unbarred $0$-colored letter of $B$, except for the last letter, and unbar each barred letter. Next we change the color of each $m$-colored  letter of $B$ to color $l-m$ for all $m=1,\ldots,l-1$, except for the last letter, but change the color of the last letter from $a$ to $l-1-a$.
 Finally for each $i$, we replace each occurrence of the $i$th smallest value in $B$ with the $i$th largest value, leaving the bars and colors intact. 

Since $a+b=l(n-j)-1$, by~\eqref{sym:flagwith}  we have $Q_{n-j,a}=Q_{n-j,b}$, which is equivalent to
\begin{equation}\label{bi:sy}
\sum_{i\geq0}Q_{n-j,a,i}(\x)=\sum_{i\geq0}Q_{n-j,b,i}(\x).
\end{equation}
For any $m,k,i$, it is not hard to see from Theorem~\ref{thm:ornament} that
\begin{equation}\label{propertyGR}
Q_{m,k,i}=h_iQ_{m-i,k,0}.
\end{equation}
Thus, Eq.~\eqref{bi:sy} becomes
$$
\sum_{i\geq0}h_{i}Q_{n-j-i,a,0}=\sum_{i\geq0}h_iQ_{n-j-i,b,0}.
$$
Multiplying both sides by $h_j$ then gives
$$
\sum_{i\geq0}h_jh_{i}Q_{n-j-i,a,0}=\sum_{i\geq0}h_jh_iQ_{n-j-i,b,0}.
$$
Applying~\eqref{propertyGR} once again, we obtain~\eqref{symsuns:sym}.
\end{proof}

\begin{remark}As the analytical proof of Theorem~\ref{sym:iden1} is reversible, the above bijective  proof together with the bijective proof of~\eqref{sym:flageul} would provide a different proof of Corollary~\ref{hy:flagexc} using  interpretations of $Q_{n,k,j}$ as colored ornaments and colored banners. In particular, this gives an alternative approach to the step~3 in~\cite[Theorem~1.2]{sw}.
\end{remark}

\vskip 0.2cm
\noindent\textbf{Theorem~\ref{sym:iden2}.}
Let $Q_{n,k}=\sum_{j}Q_{n,k,j}$. For $a,b\geq1$ such that $a+b=ln$,
\begin{equation}\label{syms:sym}
\sum_{i=0}^{n-1}h_{i}Q_{n-i,a-1}=\sum_{i=0}^{n-1}h_iQ_{n-i,b-1}.
\end{equation}

\begin{proof} Letting $r=1$ in~\eqref{eq:flagexc} we have
$$
\sum_{n,k\geq0} Q_{n,k}t^kz^n=\frac{(1-t)H(z)}{H(t^lz)-tH(z)}.
$$
Subtracting both sides by $Q_{0,0}=1$ gives
$$
\sum_{n\geq1, k\geq0} Q_{n,k}t^kz^n=\frac{H(z)-H(t^lz)}{H(t^lz)-tH(z)}.
$$
By Cross-multiplying and then identifying the coefficients of $t^kz^n$ ($1\leq k\leq ln-1$) yields
$$
\sum_{i=0}^{n-1}h_iQ_{n-i,k-li}=\sum_{i=0}^{n-1}h_iQ_{n-i,k-1}.
$$
Applying the symmetry property~\eqref{sym:flagwith} to the left side then becomes
$$
\sum_{i=0}^{n-1}h_iQ_{n-i,ln-1-k}=\sum_{i=0}^{n-1}h_iQ_{n-i,k-1},
$$
which is~\eqref{syms:sym} when $a-1=k-1$ and $b-1=ln-1-k$ since now $1\leq k\leq ln-1$.
\end{proof}

To construct a bijective proof of Theorem~\ref{sym:iden2}, we need a refinement of the decomposition of the colored banners from~\cite{hy}. We first recall some definitions therein. 

A {\em $0$-colored marked sequence}, denoted $(\omega,b,0)$, is a weakly increasing sequence $\omega$ of positive integers, together with a positive integer $b$, which we call the mark, such that $1\leq b<\length(\omega)$. The set of all $0$-colored marked sequences with $\length(\omega)=n$ and mark equal to $b$ will be denoted $M(n,b,0)$.

For $m\in[l-1]$, a {\em$m$-colored marked sequence}, denoted $(\omega,b,m)$, is a weakly increasing sequence $\omega$ of positive integers, together with a  integer $b$ such that $0\leq b<\length(\omega)$. The set of all $m$-colored marked sequences with $\length(\omega)=n$ and mark equal to $b$ will be denoted $M(n,b,m)$.



Here we will use the original colored banner interpretation, see Remark~\ref{hyatt:ban}.
Let $K_0(n,j,\Vec{\beta})$ denote the set of all colored banners of length $n$, with Lyndon type having no parts of size one formed by a $0$-colored letter, color vector equal to $\Vec{\beta}$ and $j$ bars. For $m\in[l-1]$ and $\beta_m>0$, define
$$
X_m:=\biguplus_{0\leq i\leq n-1\atop j-n+i<k\leq j}K_0(i,k,\Vec{\beta}(\hat{m}))\times M(n-i,j-k,m),
$$
where $\Vec{\beta}(\hat{m})=(\beta_1,\ldots,\beta_{m-1},\beta_m-1,\ldots,\beta_{l-1})$ and let $X_m:=0$ if $\beta_m=0$. We also define 
$$
X_0:=\biguplus_{0\leq i\leq n-2\atop j-n+i<k<j}K_0(i,k,\Vec{\beta})\times M(n-i,j-k,0).
$$
\begin{thm} \label{decomp:refine}
There is a bijection 
$$
\Upsilon : K_0(n,j,\beta)\rightarrow\biguplus_{m=0}^{l-1} X_m
$$
such that if $\Upsilon(B)=(B',(\omega,b,m))$, then $\wt(B)=\wt(B')\wt(\omega)$ and 
$\Vec{\beta}(\hat{m})=\Vec\col(B')$ if $m\geq1$ otherwise $\Vec{\beta}=\Vec\col(B')$. 
\end{thm}
\begin{proof} By~\cite[lemma~4.3]{dw}, every banner $B\in K_0(n,j,\beta)$ has a unique factorization, that we also called increasing factorization (here we admit parts of size one formed by a letter with positive color), $B=B_1\cdot B_2\cdots B_d$ where each $B_i$ has the form
$$
B_i=(\underbrace{a_i,...,a_i}_{p_i\text{ times}})\cdot u_i,
$$
where $a_i\in\mathcal{B}$, $p_i>0$ and $u_i$ is a word (possibly empty) over the alphabet $\mathcal{B}$ whose letters are all strictly less than $a_i$ with respect to $<_B$, $a_1\leq_B a_2\leq_B\cdots\leq_B a_d$ and if $u_i$  is empty  then $B_i=a_i$ and for each $k\geq i$ with $a_k=a_i$ we has $B_k=B_i=a_i$. Note that the increasing factorization is a refinement of the Lyndon factorization. 

For example, the Lyndon factorization of the banner 
$$
6^1,1^2, 5^1, 6^1, 6^1, \overline{4^0}, \overline{4^0}, 4^1, 4^0, \overline{4^0}, 3^2, 5^0, 7^1
$$
is 
$$
(6^1,1^2, 5^1)\cdot(6^1)\cdot(6^1)\cdot(\overline{4^0}, \overline{4^0}, 4^1, 4^0, \overline{4^0}, 3^2) \cdot(5^0, 7^1),
$$
and its increasing factorization is 
$$
(6^1,1^2, 5^1)\cdot(6^1)\cdot(6^1)\cdot(\overline{4^0}, \overline{4^0}, 4^1, 4^0)\cdot(\overline{4^0}, 3^2)\cdot(5^0, 7^1).
$$

First, we take the increasing factorization of $B$, say $B=B_1\cdot B_2\cdots B_d$. Let 
$$
B_d=(\underbrace{a,...,a}_{p\text{ times}})\cdot u,
$$
where $a\in\mathcal{B}$, $p>0$ and $u$ is a word (possibly empty) over $\mathcal{B}$ whose letters are all strictly less than $a$ with respect to the order $<_B$. Let  $\gamma$ be the bijection defined in~\cite[Theorem~4.5]{hy}. Now we describe the map $\Upsilon$. 
\vskip 0.1in
{\bf Case 1:} $a$ is $0$-colored. Define $\Upsilon(B)=\gamma(B)$.
\vskip 0.1in
{\bf Case 2:} $a$ has positive color and $u$ is not empty. Suppose that $u=i_1, i_2,\cdots, i_k$. 
\vskip 0.1in
{\bf Case 2.1:} 
If $k\geq2$, then define  $\omega=|i_1|$, $b=0$, $m$ is the color of $i_1$ and 
$B'=B_1\cdots B_{d-1}\cdot \widetilde{B_d}$, where 
$$\widetilde{B_d}
=\underbrace{a,...,a}_{p\text{ times}},i_2,\cdots,i_k.$$
\vskip 0.1in
{\bf Case 2.2:} 
If $k=1$, then define $\omega=|i_1|$, $b=0$, $m$ is the color of $i_1$ and 
$$B'=B_1\cdots B_{d-1}\cdot\underbrace{a\cdot a\cdots a}_{p\text{ times}},$$
where each $a$ is a factor.
\vskip 0.1in
{\bf Case 3:} $a$ has positive color and $u$ is empty. In this case $B_d=a$. Define $\omega=|a|$, $b=0$, $m$ is the color of $a$ and 
$B'=B_1\cdots B_{d-1}$.
\vskip 0.1in
This complete the description of the map $\Upsilon$. Next we describe $\Upsilon^{-1}$. Suppose we are given a banner $B$ with increasing factorization $B=B_1\cdots B_d$ where 
$$B_d=\underbrace{a,...,a}_{p\text{ times}},j_1,\cdots, j_k,$$ 
and a $m$-colored marked sequence $(\omega,b,m)$.
\vskip 0.1in
{\bf Case A:} (inverse of Case 1) $a$ is $0$-colored or $\length(\omega)\geq2$. Define 
$$\Upsilon^{-1}((B,(\omega,b,m)))=\gamma^{-1}((B,(\omega,b,m))).$$
\vskip 0.1in
{\bf Case B:} (inverse of Case 2.1) $a$ has positive color, $\omega=j_0$ is a letter with positive color $m$ and $j_1,\cdots, j_k$ is not empty. Then let $\Upsilon^{-1}((B,(\omega,b,m)))=B_1\cdots B_{d-1} \cdot\widetilde{B_d}$, where 
$$
\widetilde{B_d}=\underbrace{a,...,a}_{p\text{ times}},j_0,j_1,\cdots, j_k.
$$
\vskip 0.1in
{\bf Case C:} $a$ has positive color, $\omega=j_0$ is a letter with positive color $m$ and $B_d=a$. In this case, there exists an nonnegative integer $k$ such that $B_{d-k}=B_{d-k+1}=\cdots=B_d=a$ but $B_{d-k-1}\neq a$. 

\vskip 0.1in
{\bf Case C1:} (inverse of Case 3)  If $j_0\geq_B a$, then define 
$$
\Upsilon^{-1}((B,(\omega,b,m)))=B_1\cdots B_d\cdot j_0,
$$
where $j_0$ is a factor. 

\vskip 0.1in
{\bf Case C2:} (inverse of Case 2.2) Otherwise $j_0$ is strictly less than $a$ with respect to $<_B$ and we define $\Upsilon^{-1}((B,(\omega,b,m)))=B_1\cdots B_{d-k-1}\cdot\widetilde{B_{d-k}}$, where 
$$
\widetilde{B_{d-k}}=\underbrace{a,...,a}_{k+1\text{ times}},j_0.
$$
\vskip 0.1in
This completes the description of $\Upsilon^{-1}$. One can check case by case that both maps are well defined and in fact inverses of each other.
\end{proof}

For any nonnegative integers $i,j$, let $K_{j}(n,i,\Vec{\beta})$ denote the set of all colored banners of length $n$, with Lyndon type having $j$ parts of size one formed by a $0$-colored letter, color vector equal to $\Vec{\beta}$ and $i$ bars. Let $\Com_{j}(n,i,\Vec{\beta})$ be the set of all compositions
$$
\sigma=(\omega_0,(\omega_1,b_1,m_1),\ldots,(\omega_r,b_r,m_r))
$$
for some integer $r$, where $\omega_0$ is a weakly increasing word of positive integers of length $j$ and
each $(\omega_i,b_i,m_i)$ is a $m_i$-colored marked sequence and satisfying
$$
 \sum_{j=0}^r\length(\omega_j)=n,\quad\sum_{j=1}^r b_j=i\quad\text{and}\quad\Vec{\beta}=(\beta_1,\ldots,\beta_{l-1}),
$$
where $\beta_k$ equals the number of $m_i$ such that $m_i=k$. Define the weight of $\sigma$ by
$$
\wt(\sigma):=\wt(\omega_0)\cdots\wt(\omega_r).
$$

By Theorem~\ref{decomp:refine}, we can construct a weight preserving bijection between $K_{j}(n,i,\Vec{\beta})$ and $\Com_{j}(n,i,\Vec{\beta})$ by first factoring out  the $j$ parts of size one formed by a $0$-colored letter in the Lyndon factorization of a banner and then factoring out  marked sequences step by step in the increasing factorization of the remaining banner. Thus we have the following interpretation of $Q_{n,k,j}$.

\begin{cor}\label{compo:inter}
We have 
$$
Q_{n,k,j}=\sum_{i\in\N,\Vec{\beta}\in\N^{l-1}\atop{\sigma\in\Com_j(n,i,\Vec{\beta})\atop li+\cs(\Vec{\beta})=k}}\wt(\sigma).
$$
\end{cor}


\begin{definition}
For each fixed positive integer $n$, a {\em two-fix-banner} of length $n$ is a sequence
\begin{equation}\label{two:banner}
\v=(\omega_0,(\omega_1,b_1,m_1),\ldots,(\omega_r,b_r,m_r),\omega'_0)
\end{equation}
satisfying the following conditions:
\begin{itemize}
\item[(C1)] $\omega_0$ and $\omega'_0$ are two weakly increasing sequences of positive integers, possibly empty;
\item[(C2)] each $(\omega_i,b_i,m_i)$ is a $m_i$-colored marked sequence;
\item[(C3)] $\length(\omega_0)+\length(\omega_1)+\cdots+\length(\omega_r)+\length(\omega'_0)=n$.
\end{itemize}
Define the {\em flag excedance} statistic of $\v$ by 
$$
\fexc(\v):=l\sum_{i=1}^r b_i +\sum_{i=1}^rm_r.
$$
Let $\TB_n$ denote the set of all two-fix-banners of length $n$.
\end{definition}

\begin{proof}[{\bf A bijective proof of Theorem~\ref{sym:iden2}}] The two-fix-banner $\v$ in~\eqref{two:banner} is in bijection with the pair $(\sigma,\omega)$, where $\omega=\omega'_0$ is a weakly increasing sequence of positive integers with length $i$ for some nonnegative integer $i$ and  $\sigma=(\omega_0,(\omega_1,b_1,m_1),\ldots,(\omega_r,b_r,m_r))$ is a composition with 
$\sum_{j=0}^r\length(\omega_j)=n-i$. Thus, by Corollary~\ref{compo:inter} we obtain the following interpretation.
\begin{lem}For any nonnegative integer $a$, we have
$$
\sum_{\v\in\TB_n\atop\fexc(\v)=a}\wt(\v)=\sum_{i=0}^{n-1}h_iQ_{n-i,a}.
$$
\end{lem}

By the above lemma, it suffices to construct an involution $\Phi :\TB_n\rightarrow\TB_n$ satisfying $\fexc(\v)+\fexc(\Phi(\v))=ln-2$ for each $\v\in\TB_n$. First we need to define two local involutions. For a weakly increasing sequence of positive integers $\omega$ with $\length(\omega)=k$, we define
$$
d'(\omega)=(\omega,k-1,l-1),
$$
which is a $(l-1)$-colored marked sequence. For a $m$-colored mark sequence $(\omega,b,m)$ with $\length(\omega)=k$, we define 
$$
d((\omega,b,m))=
\begin{cases}
(\omega,k-b,0) &\text{if $m=0$;}\\
(\omega,k-1-b,l-m),&\text{otherwise.}
\end{cases}
$$
We also define
$$
d'((\omega,b,m))=
\begin{cases}
\omega, &\text{if $b=k-1$ and $m=l-1$;}\\
(\omega,k-1-b,l-1-m),&\text{otherwise.}
\end{cases}
$$
One can check that $d$ and $d'$ are well-defined involutions. 

Let $\v$ be a two-fix-banner and write
$$\v=(\tau_0, \tau_1, \tau_2, \ldots, \tau_{r-1}, \tau_{r}, \tau_{r+1}),$$
where  $\tau_0=\omega_0$ and $\tau_{r+1}=\omega'_0$.
If $\tau_i$ (respectively $\tau_j$) is  the leftmost (respectively rightmost) non-empty
sequence (clearly $i=0, 1$ and $j=r, r+1$), we can write $\v$ in the following compact way by removing the empty sequences at the beginning or at the end:
\begin{equation}
\label{eq:v_compacted}
{\bf v}=(\tau_i, \tau_{i+1},  \ldots, \tau_{j-1}, \tau_{j}).
\end{equation}
It is easy to see that the above procedure is reversible 
by adding some necessary empty words at the two  ends of the compact form~\eqref{eq:v_compacted}. Now we work with the compact form.

If $i=j$ ($\v$ has only one sequence), 
we define 
$$\Phi(\v)=
\begin{cases}
(\emptyset, (\tau_i,n-1,l-2), \emptyset), &\text{if $\tau_i$ is a weakly increasing sequence;}\\
( \omega, \emptyset), &\text{if $\tau_i=(\omega,n-1,l-2)$ is a marked sequence;}\\
(\emptyset, (\omega,n-1-b,l-2-m), \emptyset), &\text{otherwise, suppose $\tau_i=(\omega,b,m)$.}
\end{cases}
$$

If $j>i$ ($\v$ has at least two sequences),  we define the two-fix-banner $\Phi(\v)$ by
$$
\Phi(\v)=(d'(\tau_i), d(\tau_{i+1}), d(\tau_{i+2}),  \ldots, d(\tau_{j-1}), d'(\tau_{j})).
$$

As $d$ and $d'$ are involutions, $\Phi$ is also an involution and one can check that in both cases $\Phi$ satisfy the desired property. This completes our bijective proof. 
\end{proof}

By Lemma~\ref{lem:psflag}, if we apply $\ps$ to both sides of~\eqref{symsuns:sym} and~\eqref{syms:sym} then we obtain the following two symmetrical $q$-Eulerian identities.

\begin{cor}
For $a,b\geq1$ and $j\geq0$ such that $a+b+1=l(n-j)$,
\begin{align}\label{sym:col2}
\sum_{k\geq0}{\bmatrix n\\ k\endbmatrix}_{q}A^{(l)}_{k,a,j}(q)
=\sum_{k\geq0}{\bmatrix n\\ k\endbmatrix}_{q}A^{(l)}_{k,b,j}(q).
\end{align}
\end{cor}

\begin{cor}
For $a,b\geq1$ such that $a+b=ln$,
\begin{equation} \label{sym:col1}
\sum_{k\geq1}{\bmatrix n\\ k\endbmatrix}_{q}A^{(l)}_{k,a-1}(q)
=\sum_{k\geq1}{\bmatrix n\\ k\endbmatrix}_{q}A^{(l)}_{k,b-1}(q).
\end{equation}
\end{cor}

\subsection{Two interpretations of colored $(q,r)$-Eulerian polynomials} 
We will introduce the colored hook factorization of a colored permutation.
A word $w=w_1w_2\ldots w_m$ over $\N$ is called a \emph{hook} if $w_1>w_2$ and either $m=2$, or $m\geq 3$ and $w_2<w_3<\ldots<w_m$. We can extend the hooks to colored hooks. Let 
$$
[n]^l\subset\N^l:=\left\{1^{0}, 1^{1}, \ldots, 1^{l-1},
2^{0}, 2^{1}, \ldots, 2^{l-1},\ldots,
i^0, i^1, \ldots, i^{l-1},\ldots\right\}.
$$
A word $w=w_1w_2\ldots w_m$ over $\N^l$ is called a {\em colored hook} if 
\begin{itemize}
\item  $m\geq2$ and $|w|$ is a hook with only $w_1$ may have positive color;
\item or $m\geq1$ and $|w|$ is an increasing word and only $w_1$ has positive color.
\end{itemize}
Clearly, each colored permutation $\pi =\pi_1\pi_2\dots \pi_n\in C_l\wr\S_n$ admits a unique factorization, called its \emph{colored hook factorization}, $p\tau_{1}\tau_{2}. . . \tau_{r}$, where $p$ is a word formed by $0$-colored letters, $|p|$ is an increasing word over $\N$ and each factor $\tau_1$, $\tau_2$, \ldots, $\tau_k$ is a colored hook. 
To derive the colored hook factorization of a colored permutation, one can start from the right and factor out each colored hook step by step. When $l=1$, colored hook factorization is the hook factorization introduced by Gessel~\cite{ge} in his study of the derangement numbers.

For example, the colored hook factorization of 
\begin{equation}\label{hook:exam}
2^0\,4^0\,5^1\,8^0\,3^0\, 7^0\,10^1\,1^0\,9^0\,6^1\in C_2\wr\S_{10}
\end{equation}
is 
$$
2^0\,4^0\,|5^1\,|8^0\,3^0\, 7^0\,|10^1\,1^0\,9^0\,|6^1.
$$

Let $w=w_1w_2\ldots w_m$ be a word over $\N$. Define 
$$
\inv(w):=|\{(i,j) : i<j, w_i>w_j\}|.
$$
For a colored permutation $\pi=\in C_l\wr\S_n$ with colored  hook factorization $p\tau_{1}\tau_{2}. . . \tau_{r}$, we define 
$$
\inv(\pi):=\inv(|\pi|)\quad\text{and}\quad\lec(\pi):=\sum_{i=1}^r\inv(|\tau_i|).
$$
We also define 
$$
\flec(\pi):=l\cdot\lec(\pi)+\sum_{i=1}^n\epsilon_i\quad\text{and}\quad\pix(\pi):=\length(p).
$$
For example, if $\pi$ is the colored permutation in~\eqref{hook:exam}, then $\inv(\pi)=16$, $\lec(\pi)=4$, $\flec(\pi)=11$ and $\pix(\pi)=2$.

Through some similar calculations as~\cite[Theorem~4]{fh1}, we can prove the following interpretation of the colored $(q,r)$-Eulerian polynomial $A_n^{(l)}(t,r,q)$. 
\begin{thm}\label{th:hook}
 For $n\geq1$, we have
$$
A_n^{(l)}(t,r,q)=\sum_{\pi\in C_l\wr\S_n} t^{\flec(\pi)}r^{\pix(\pi)}q^{\inv(\pi)-\lec(\pi)}.
$$
\end{thm}
\begin{proof}
The proof is very similar to the proof of~\cite[Theorem~4]{fh1}, which is the $l=1$ case of the theorem. The details are omitted.
\end{proof}

The \emph{Eulerian differential operator} $\delta_{x}$ used below is defined by
$$
\delta_{x}(f(x)):=\frac{f(x)-f(qx)}{x},
$$
for any $f(x)\in\Q[q][[x]]$ in the ring of formal power series in $x$ over $\Q[q]$.
%
The recurrence  in~\cite[Theorem~2]{lin} can be generalized to the colored $(q,r)$-Eulerian polynomials as follows.

\begin{thm}\label{recu}
The colored $(q,r)$-Eulerian polynomials satisfy the following recurrence formula: 
\begin{align}\label{recurrence2}
A_{n+1}^{(l)}(t,r,q)=(r+t[l-1]_tq^n)A_n^{(l)}(t,r,q)+t[l]_t\sum_{k=0}^{n-1}{n\brack k}_{q}q^{k}A_k^{(l)}(t,r,q)A_{n-k}^{(l)}(t,q)
\end{align}
with $A_0^{(l)}(t,r,q)=1$ and $A_1^{(l)}(t,r,q)=r$.
\end{thm}
\begin{proof}
 It is not difficult to show that, for any variable $y$,
$\delta_{z}(e(yz;q))=ye(yz;q)$.
Now, applying $\delta_{z}$ to both sides of~\eqref{expo-color} and using the above property and~\cite[Lemma~7]{lin}, we obtain
\begin{align*}
&\sum_{n\geq0}A^{(l)}_{n+1}(t,r,q)\frac{z^n}{(q;q)_n}=\delta_{z}\left(\frac{(1-t)e(rz; q)}{e(t^lz; q)-te(z;q)}\right)=\\
=&\delta_{z}((1-t)e(rz; q))(e(t^lz; q)-te(z;q))^{-1}+\delta_{z}\left((e(t^lz; q)-te(z;q))^{-1}\right)(1-t)e(rzq^l; q)\\
=&\frac{r(1-t)e(rz; q)}{e(t^lz; q)-te(z;q)}+\frac{(1-t)e(rzq; q)(te(z; q)-t^le(t^lz; q))}{(e(t^lqz; q)-te(qz;q))(e(t^lz; q)-te(z;q))}\\
=&\frac{r(1-t)e(rz; q)}{e(t^lz; q)-te(z;q)}+\frac{(1-t)e(rzq; q)}{e(t^lqz; q)-te(qz;q)}\left(\frac{t^le(z; q)-t^le(t^lz; q)}{e(t^lz; q)-te(z;q)}+\frac{te(z; q)-t^le(z; q)}{e(t^lz; q)-te(z;q)}\right)\\
=&\left(\sum_{n\geq0}A_{n}^{(l)}(t,r,q)\frac{(qz)^n}{(q;q)_n}\right)\left((t+\cdots+t^{l-1})\sum_{n\geq0}A_{n}^{(l)}(t,q)\frac{z^n}{(q;q)_n}+t^l\sum_{n\geq1}A_{n}^{(l)}(t,q)\frac{z^n}{(q;q)_n}\right)\\
&+r\sum_{n\geq0}A_{n}^{(l)}(t,r,q)\frac{z^n}{(q;q)_n}.
\end{align*}
Taking the coefficient of $\frac{z^n}{(q;q)_n}$ in both sides of  the above equality, we get~\eqref{recurrence2}. 
\end{proof}

\begin{remark}
Once again, the polynomial $d_n^B(t)$ is $t$-symmetric with center of symmetry $n$ and $t$-unimodal follows from the recurrence~\eqref{recurrence2} by induction on $n$ using the fact in Lemma~\ref{fact:unimodal}.
\end{remark}

The above recurrence formula enables us to obtain another interpretation of the colored $(q,r)$-Eulerian polynomials $A_n^{(l)}(t,r,q)$. First we define the {\em absolute descent number} of a colored permutation $\pi\in C_l\wr\S_n$, denoted $\des^{\Abs}(\pi)$, by 
$$
\des^{\Abs}(\pi):=|\{i\in[n-1] : \epsilon_i=0\,\,\text{and}\,\, |\pi_i|>|\pi_{i+1}|\}|.
$$
We also define the {\em flag absolute descent number} by
$$
\fdes^{\Abs}(\pi):=l\cdot\des^{\Abs}(\pi)+\sum_{i=1}^n\epsilon_i.
$$
A \emph{colored admissible inversion} of $\pi$ is a pair $(i,j)$ with $1\leq i<j\leq n$ that satisfies any one of the following three conditions 
 \begin{itemize}
 \item $1<i$ and $|\pi_{i-1}|<|\pi_{i}|>|\pi_j|$;
 \item there is some $k$ such that $i<k<j$ and $|\pi_j|<|\pi_i|<|\pi_k|$;
 \item $\epsilon_j>0$ and for any $k$ such that $i\leq k<j$, we have $|\pi_k|<|\pi_j|<|\pi_{j+1}|$, where we take the convention $|\pi_{n+1}|=+\infty$.
 \end{itemize}
 We write $\ai(\pi)$ the number of colored admissible inversions of $\pi$. For example, if $\pi=4^0\,1^0\,2^1\,5^0\,3^1$ in $C_2\wr\S_5$, then $\ai(\pi)=3$.
 When $l=1$, colored admissible inversions agree with admissible inversions introduced by  Linusson,  Shareshian and Wachs~\cite{lsw} in their study of poset topology. 
 
 Finally, we define a statistic, denoted by ``$\rix$", on the set of all words over $\N^l$ recursively.
Let $w=w_1\cdots w_n$ be a word over $\N^l$.  Suppose that $w_i$ is the unique rightmost element of $w$ such that $|w_i|=\max\{|w_1|,|w_2|,\ldots,|w_n|\}$.  We define $\rix(w)$ by (with convention that $\rix(\emptyset)=0$)
$$ \rix(w):=
 \begin{cases}
0,&\text{if $i=1\neq n$,}\\
1+\rix(w_1\cdots w_{n-1}),&\text{if $i=n$ and $\epsilon_n=0$},\\
\rix(w_{i+1}w_{i+2}\cdots w_n),& \text{if $1<i<n$.}
\end{cases}
$$
As  a colored permutation can be viewed as a word over $\N^l$, the statistic $\rix$ is well-defined on colored permutations. For example, if $\pi=1^0\,6^1\,2^0\,5^1\,3^0\,4^1\,7^0\in C_2\wr\S_7$, then $\rix(\pi)=1+\rix(1^0\,6^1\,2^0\,5^1\,3^0\,4^1)=1+\rix(2^0\,5^1\,3^0\,4^1)=1+\rix(3^0\,4^1)=1+\rix(3^0)=2$.

\begin{cor}
For $n\geq1$, we have
\begin{equation}\label{admiss:inv}
A_n^{(l)}(t,r,q)=\sum_{\pi\in C_l\wr\S_n} t^{\fdes^{\Abs}(\pi)}r^{\rix(\pi)}q^{\ai(\pi)}.
\end{equation}
\end{cor}
\begin{proof}
By considering the position of the element of $\pi$ with maximal absolute value, we can show that the right hand side of~\eqref{admiss:inv} satisfies the same recurrence formula~\eqref{recurrence2} and initial conditions as $A_n^{(l)}(t,r,q)$. The discussion is quite similar to the proof of~\cite[Theorem~8]{lin} and thus is left to the interested reader. 
\end{proof}

By setting $r=1$ in~\eqref{admiss:inv}, we have 
$
A_n^{(l)}(t,q)=\sum_{\pi\in C_l\wr\S_n} t^{\fdes^{\Abs}(\pi)}q^{\ai(\pi)}
$.
Another statistic whose joint distribution with $\fdes^{\Abs}$ is the same as that of $\ai$ will be discussed in next section (see Corollary~\ref{three:stats}).

\section{Rawlings major index for colored permutations}
\label{color:rawling}
\subsection{Rawlings major index and colored Eulerian quasisymmetric functions} 
For $\pi\in C_l\wr\S_n$ and $k\in[n]$, we define 
\begin{align*}
&\Des_{\geq k}(\pi):=\{i\in[n] : |\pi_i|>|\pi_{i+1}|\,\,\text{and either}\,\,\epsilon_i\not=0\,\, \text{or}\,\, |\pi_i|-|\pi_{i+1}|\geq k\},\\
&\inv_{<k}(\pi):=|\{(i,j)\in[n]\times[n]: i<j, \epsilon_i=0\,\,\text{and}\,\,0<|\pi_i|-|\pi_{j}|<k\}|,\\
&\maj_{\geq k}(\pi):=\sum_{i\in\Des_{\geq k}(\pi)}i.
\end{align*}
Then the
{\em Rawlings major index} of $\pi$ is defined as
$$
\rmaj_k(\pi):=\maj_{\geq k}(\pi)+\inv_{<k}(\pi).
$$
For example, if $\pi=2^0\,6^1\,1^0\,5^0\,4^1\,3^1\,7^0\in C_2\wr\S_7$, then $\Des_{\geq 2}(\pi)=\{2,4\}$, $\inv_{<2}(\pi)=2$, $\maj_{\geq 2}(\pi)=2+4=6$ and so $\rmaj_2(\pi)=6+2=8$.
Note that when $l=1$, $\rmaj_k$ is the $k$-major index studied by Rawlings~\cite{ra}.

 Let  $Q_{n,k,\Vec{\beta}}$ be the {\em colored Eulerian quasisymmetric functions} defined  by
$$
Q_{n,k,\Vec{\beta}}:=\sum_{\Vec{\alpha}}Q_{n,k,\Vec{\alpha},\Vec{\beta}}.
$$
The main result of this section is the following interpretation of $Q_{n,k,\Vec{\beta}}$.
\begin{thm}\label{inter:des2}
We have
\begin{equation*}
Q_{n,k,\Vec{\beta}}=\sum_{\inv_{<2}(\pi)=k\atop\Vec{\col}(\pi)=\Vec{\beta}}F_{n,\Des_{\geq2}(\pi)}.
\end{equation*}
\end{thm}

It follows from Theorem~\ref{inter:des2} and Eq.~\eqref{DEX:int} that 
\begin{equation*}\label{exc:des2}
\sum_{\inv_{<2}(\pi)=k\atop\Vec{\col}(\pi)=\Vec{\beta}}F_{n,\Des_{\geq2}(\pi)}=\sum_{\exc(\pi)=k\atop\Vec{\col}(\pi)=\Vec{\beta}}F_{n,\Dex(\pi)}.
\end{equation*}
By Lemma~\ref{DEX:lem} and Eq.~\eqref{quasi-ps}, if we apply $\ps$ to both sides of the above equation, we will obtain the following new interpretation of the colored $q$-Eulerian polynomial $A_{n}^{(l)}(t,q)$.

\begin{cor}\label{three:stats}
Let $s^{\Vec{\beta}}=s_1^{\beta_1}\cdots s_{l-1}^{\beta_{l-1}}$ for $\Vec{\beta}\in\N^{l-1}$. Then
$$
\sum_{\pi\in C_l\wr\S_n}t^{\exc(\pi)}q^{\maj(\pi)}s^{\Vec{\col}(\pi)}=\sum_{\pi\in C_l\wr\S_n}t^{\inv_{<2}(\pi)}q^{\rmaj_{2}(\pi)}s^{\Vec{\col}(\pi)}.
$$
Consequently,
$$
A_n^{(l)}(t,q)=\sum_{\pi\in C_l\wr\S_n}t^{l\cdot\inv_{<2}(\pi)+\sum_{i=1}^n\epsilon_i}q^{\maj_{\geq2}(\pi)}=\sum_{\pi\in C_l\wr\S_n}t^{\fdes^{\Abs}(\pi)}q^{\maj_{\geq2}(\pi^{-1})}.
$$
\end{cor}
\begin{remark}
When $l=1$, the above result reduces to~\cite[Theorem~4.17]{sw3}.
In view of interpretation~\eqref{admiss:inv}, an interesting open problem (even for $l=1$) is to describe a statistic, denoted $\fix_2$, equidistributed with $\fix$ so that 
$$
A_n^{(l)}(t,r,q)=\sum_{\pi\in C_l\wr\S_n}t^{\fdes^{\Abs}(\pi)}r^{\fix_2(\pi)}q^{\maj_{\geq2}(\pi^{-1})}.
$$
\end{remark}
 

%


\subsection{Proof of Theorem~\ref{inter:des2}: Chromatic quasisymmetric functions}
Let $G$ be a graph with vertex set $[n]$ and edge set $E(G)$.
A \emph{coloring}  of a graph $G$ is a function $\k: [n]\rightarrow \P$ such that  whenever $\{i,j\}\in E(G)$ we have $\k(i)\neq\k(j)$. 
Given a function $\k: [n]\rightarrow \P$, set 
$$
\x_{\k}:=\prod_{i\in[n]}x_{\k(i)}.
$$
Shareshian and Wachs~\cite{sw3} generalized Stanley's Chromatic symmetric function of $G$ to the \emph{Chromatic quasisymmetric function} of $G$ as 
$$
X_G({\bf x},t):=\sum_{\k}t^{\asc_G(\k)}\x_{\k},
$$
where the sum is over all colorings $\k$ and
$$
\asc_G(\k):=|\{\{i,j\}\in E(G) : i<j\,\, \text{and}\,\, \k(i)>\k(j)\}|.
$$

Recall that an \emph{orientation} of $G$ is a directed graph $\o$ with the same vertices, so that for every edge $\{i,j\}$ of $G$, exactly one of $(i,j)$ or $(j,i)$ is an edge of $\o$. An orientation is often regarded as giving a direction to each edge of an undirected graph.

Let $P$ be a poset. Define
$$
X_p=X_P(\x):=\sum_{\sigma}\x_{\sigma},
$$
summed over all  strict order-reversing maps $\sigma : P\rightarrow\P$ (i.e. if $s<_Pt$, then 
$\sigma(s)>\sigma(t)$).
Let $\o$ be an acyclic orientation of $G$ and $\k$ a  coloring. We say that $\k$ is \emph{$\o$-compatible} if $\k(i)<\k(j)$ whenever $(j,i)$ is an edge of $\o$. Every proper coloring is compatible with exactly one acyclic orientation $\o$, viz., if $\{i,j\}$ is an edge of $G$ with $\k(i)<\k(j)$, then let $(j,i)$ be an edge of $\o$. Thus if $K_{\o}$ denotes the set of $\o$-compatible colorings of $G$, and if $K_G$ denotes the set of all colorings of $G$, then we have a disjoint union $K_G=\cup_{\o}K_{\o}$. Hence $X_G=\sum_{\o}X_{\o}$, where $X_{\o}=\sum_{\k\in K_{\o}}\x_{\k}$. Since $\o$ is acyclic, it induces a poset $\bar{\o}$: make $i$ less than $j$ if $(i,j)$ is an edge of $\o$ and then take the transitive closure of this relation. By the definition of $X_P$ for a poset and of $X_{\o}$ for an acyclic orientation, we have $X_{\bar{\o}}=X_{\o}$. Also, according to the definition of $\asc_G(\k)$ for a  \emph{$\o$-compatible} coloring $\k$, $\asc_G(\k)$ depends only on $\o$, that is, 
$$
\asc_G(\k)=\asc_G(\k')\quad\text{for any $\k,\k'\in K_{\o}$}.
$$
Thus we can define  $\asc_G(\o)$ of an acyclic orientation $\o$ by
$$
\asc_G(\o):=\asc_G(\k)\quad\text{for any $\k\in K_{\o}$}.
$$ 
So
\begin{equation}\label{acyclic}
X_G(\x,t)=\sum_{\o}t^{\asc_G(\o)}X_{\bar{\o}},
\end{equation}
summed over all acyclic orientations of $G$.

We have the following reciprocity theorem for chromatic quasisymmetric functions, which is a refinement of Stanley~\cite[Theorem~4.2]{st}.
\begin{thm}[Reciprocity theorem] \label{recipro}
Let $G$ be a graph on $[n]$.
Define
$$
\overline{X}_G(\x,t)=\sum_{(\o,\k)}t^{\asc_G(\o)}\x_{\k},
$$
summed over all pairs $(\o,\k)$ where $\o$ is an acyclic orientation of $G$ and $\k$ is a function $\k : [n]\rightarrow \P$ satisfying $\k(i)\leq\k(j)$ if $(i,j)$ is an edge of $\o$. Then
\begin{equation*}\label{eq:reciprocity}
\overline{X}_G({\bf x},t)=\omega X_G({\bf x},t),
\end{equation*}
where $\omega$ is the involution defined at the end of the introduction.
\end{thm}

\begin{proof} For a poset $P$, define 
$$
\overline{X}_P=\sum_{\sigma} x_{\sigma},
$$
summed over all order-preserving functions  $\sigma : P\rightarrow\P$, i.e., if $s<_Pt$ then 
$\sigma(s)\leq\sigma(t)$. The reciprocity theorem for $P$-partitions~\cite[Theorem~4.5.4]{st0} implies that 
$$
\omega X_P=\overline{X}_P.
$$
Now apply $\omega$ to Eq.~\eqref{acyclic}, we get
$$
\omega X_G(\x,t)=\sum_{\o}t^{\asc_G(\o)}\omega X_{\bar{\o}}=\sum_{\o}t^{\asc_G(\o)}\overline{X}_{\bar{\o}},
$$
where $\o$ summed over all acyclic orientations of $G$. Hence $\overline{X}_G({\bf x},t)=\omega X_G({\bf x},t)$, as desired.
\end{proof}

 For $\pi \in \S_n$, the {\em $G$-inversion number} of $\pi$ is
$$
\inv_G(\pi):=|\{ (i,j) : i<j, \,\,\pi(i) > \pi(j) \mbox{ and } \{\pi(i),\pi(j) \} \in E(G) \}|.
$$
For $\pi \in \S_n$ and $P$ a poset on $[n]$, the {\em $P$-descent set} of $\pi$ is 
$$ \Des_P(\pi ):= \{ i \in [n-1] : \pi(i) >_P \pi(i+1) \}.$$ 
Define the {\em incomparability graph} $\inc(P)$  of a poset $P$ on $[n]$ to be  the graph with vertex set $[n]$ and edge set  $\{\{a,b\} : a \not\le_P b \mbox{ and } b \not\le_P a \}$. 

Shareshian and Wachs~\cite[Theorem~4.15]{sw3}
stated the following
fundamental quasisymmetric function basis decomposition of the Chromatic quasisymmetric functions, which refines the result of Chow~\cite[Corollary~2]{ch}. 
 
\begin{thm}[Shareshian-Wachs] \label{main2} Let  $G$ be the incomparability graph of a  poset $P$ on $[n]$.   Then 
$$\omega X_G(\x,t) = \sum_{\pi \in \S_n} t^{\inv_G(\pi)} F_{n,\Des_{P}(\pi)}.$$
\end{thm}

%
%
%
%
%

\begin{proof}[{\bf Proof of Theorem~\ref{inter:des2}}]
We will use the colored banner interpretation of $Q_{n,k,\Vec{\beta}}$.
Let $c=c_1c_2\ldots c_n$  be a word of length $n$ over $\{0\}\cup[l-1]$.  Define $P^{c}_{n,k}$ to be the poset on vertex set $[n]$ such that $i<_P j$ in $P^{c}_{n,k}$ if and only if $i<j$ and either $c_i\neq0$ or $j-i\geq k$. Let $G^c_{n,k}$ be the incomparability graph of $P^{c}_{n,k}$. 
\begin{figure}[h]
\setlength {\unitlength} {1mm}
\begin {picture} (98,15) \setlength {\unitlength} {1.2mm}
\thinlines
\put(8,5){\circle*{1.3}}\put(16,5){\red{\circle*{1.3}}}
\put(8,5){\line(1,0){8}}

\put(7,1){$1$}\put(15,1){$2$}\put(23,1){$3$}\put(31,1){$4$}\put(39,1){$5$}\put(47,1){$6$}
\put(55,1){$7$}\put(63,1){$8$}\put(71,1){$9$}

\put(24,5){\circle*{1.3}}\put(32,5){\circle*{1.3}}\put(40,5){\blue{\circle*{1.3}}}
\put(24,5){\line(1,0){16}}

\put(48,5){\circle*{1.3}}\put(56,5){\red{\circle*{1.3}}}
\put(48,5){\line(1,0){8}}

\put(64,5){\blue{\circle*{1.3}}}\put(72,5){\red{\circle*{1.3}}}

\end{picture}
\caption{The graph $G^c_{9,2}$ with $c=0\red{1}00\blue{2}0\red{1}\blue{2}\red{1}$.}
\label{t1}
\end{figure}

It is not difficult to see that
$$
\overline{X}_{G^c_{n,2}}({\bf x},t)=\sum_{B}\wt(B),
$$
where the sum is over all colored banners $B$ such that $B(i)$ is $c_i$-colored. Theorems~\ref{recipro}  and~\ref{main2} together gives
$$
\overline{X}_{G^c_{n,2}}(\x,t) = \sum_{\pi \in \S_n} t^{\inv_{G^c_{n,2}}(\pi)} F_{n,\Des_{P^c_{n,2}}(\pi)},
$$
which would finish the proof once we can verify that
for $\pi\in C_l\wr\S_n$, 
\begin{equation}\label{gra:rawl}
\inv_{<k}(\pi)=\inv_{G^{c}_{n,k}}(|\pi|)\quad\text{and}\quad\Des_{\geq k}(\pi)=\Des_{P^{c}_{n,k}}(|\pi|)
\end{equation}
if $c=c_1c_2\ldots c_n$ is defined by  the identification 
$$
\{1^{c_1},2^{c_2},\ldots,n^{c_n}\}=\{\pi_1,\pi_2,\ldots,\pi_n\}.
$$
\end{proof}

\subsection{Mahonian statistics on colored permutation groups} A statistic $\st$ on the colored permutation group $C_l\wr\S_n$ is called {\em Mahonian} if 
$$
\sum_{\pi\in C_l\wr\S_n}q^{\st(\pi)}=[l]_q[2l]_q\cdots[nl]_q.
$$
The {\em flag major index} of a colored permutation $\pi$, denoted $\fmaj(\pi)$, is 
$$
\fmaj(\pi):=l\cdot\maj(\pi)+\sum_{i=1}^n\epsilon_i.
$$
It is known (cf.~\cite{fr}) that the flag major index is Mahonian. 
Note that 
$$
\sum_{\pi\in C_l\wr\S_n}t^{\fexc(\pi)}r^{\fix(\pi)}q^{\fmaj(\pi)}=A_n^{(l)}(tq,r,q^l).
$$
When $l=1$, $\rmaj_k$ is Mahonian for each $k$ (see~\cite{ra}). Define the {\em flag Rawlings major index} of $\pi\in C_l\wr\S_n$, $\fmaj_k(\pi)$, by 
$$
\fmaj_k(\pi):=l\cdot\rmaj_k(\pi)+\sum_{i=1}^n\epsilon_i.
$$
We should note that $\fmaj\neq\fmaj_1$ if $l\geq2$.
By Corollary~\ref{three:stats} we see that $\fmaj_2$ is equidistributed with $\fmaj$ on colored permutation groups and thus is also Mahonian.
More general, we have the following result.
\begin{thm}
The flag Rawlings major index $\fmaj_k$ is Mahonian for any $l,k\geq1$.
\end{thm}

\begin{proof}
A poset $P$ on $[n]$ satisfies the following conditions
\begin{itemize}
\item[(1)] $x<_Py$ implies $x<_{\N}y$
\item[(2)] if the disjoint union (or direct sum) $\{x<_Pz\}+\{y\}$ is an induced subposet of $P$ then $x<_{\N}y<_{\N}z$
\end{itemize}
is called a natural unit interval order. It was observed in~\cite{sw3} that if $P$ is a natural unit interval order, then by a result of Kasraoui~\cite[Theorem~1.8]{ka},
\begin{equation}\label{unit:interval}
\sum_{\pi\in\S_n}q^{\inv_{\inc(P)}(\pi)+\maj_P(\pi)}=[1]_q\cdots[n]_q
\end{equation}
with $\maj_P(\pi):=\sum_{i\in\Des_P(\pi)}i$. Denote by $W_n^l$ the set of all words of length $n$ over $\{0\}\cup[l-1]$. For each  $c=c_1\cdots c_n$ in $W_n^l$, let $P^{c}_{n,k}$ and $G^{c}_{n,k}$ be the poset and the graph defined in the proof of Theorem~\ref{inter:des2}, respectively. Note that $P^{c}_{n,k}$ is a natural unit interval order. Thus
\begin{align*}
\sum_{\pi\in C_l\wr\S_n}q^{\fmaj_k}&=\sum_{\pi\in C_l\wr\S_n}q^{l(\rmaj_{\geq k}(\pi)+\inv_{<k}(\pi))+\sum_{i=1}^n\epsilon_i}\\
&=\sum_{ c\in W_n^l}\sum_{\pi\in\S_n}q^{l(\inv_{G^{c}_{n,k}}(\pi)+\maj_{P^{c}_{n,k}}(\pi))+\sum_i c_i}\qquad\quad\text{(by~\eqref{gra:rawl})}\\
&=\sum_{ c\in W_n^l} q^{\sum_i c_i}[1]_{q^l}\cdots[n]_{q^l}\qquad\quad\text{(by~\eqref{unit:interval})}\\
&=(1+q+\cdots+q^{l-1})^n[1]_{q^l}\cdots[n]_{q^l}=[l]_q[2l]_q\cdots[nl]_q.
\end{align*}
\end{proof}

\subsection*{Acknowledgement}
The author would like to thank Jiang Zeng for numerous remarks and advice.

\end{document}